\documentclass[12pt]{article}
\usepackage{fullpage}
\usepackage{mdwlist}
\usepackage{cite}

\usepackage{amsmath}
\usepackage{amssymb}
\usepackage{amsthm}

\usepackage[final]{graphicx}
\usepackage{color}
\usepackage{subfigure}
\usepackage{epstopdf}

\usepackage{hyperref}

\newtheorem{theorem}{Theorem}[section]
\newtheorem{proposition}[theorem]{Proposition}
\newtheorem{lemma}[theorem]{Lemma}

\theoremstyle{definition}
\newtheorem{defn}[theorem]{Definition}

\theoremstyle{remark}
\newtheorem{remark}[theorem]{Remark}
\newtheorem{example}[theorem]{Example}

\numberwithin{equation}{section}
\numberwithin{figure}{section}

\title{The Calculation and Simulation of the Price of Anarchy for Network Formation Games}
\author{Shaun Lichter, Christopher Griffin, and Terry Friesz}
\date{\today}

\begin{document}
\maketitle
\begin{abstract}
We model the formation of networks as the result of a game where by players act selfishly to get the portfolio of links they desire most.  The integration of player strategies into the network formation model is appropriate for organizational networks because in these smaller networks, dynamics are not random, but the result of intentional actions carried through by players maximizing their own objectives.  This model is a better framework for the analysis of influences upon a network because it integrates the strategies of the players involved.  We present an Integer Program that calculates the price of anarchy of this game by finding the worst stable graph and the best coordinated graph for this game.  We simulate the formation of the network and calculated the simulated price of anarchy, which we find tends to be rather low.
\end{abstract}

\section{Introduction}
In recent years, there has been extensive efforts to collect large amounts of data related to underlying social networks in order to characterize the social network in some useful way \cite{Doz11}. Social network analysis (and its various branches) is a large subject with several devoted texts e.g., \cite{CSW05,KY08}. It is impossible to provide a complete literature survey in a short space and most of the papers are not necessarily germane to the topics discussed in this paper. Nevertheless, we provide a small sampling of the literature to illustrate the breadth of research in social network analysis. \cite{AH10} suggests methods for predicting future events based on social media, while \cite{BDK07} is specifically interested in the identification of specific anonymous individuals within a given social network. \cite{FCF11} investigates the problem of mathematically modeling crowd-sourcing. This work is extended in \cite{OT11}. There is substantial interest from the statistical physics community in this problem in the form of Network Science (see the sequel). \cite{JGN01} investigate models for the growth of social networks. Very little of this work uses game theoretic principles as we do in this paper, however, \cite{SSG11} investigates the problem of a user deciding to join a social network from a game theoretic perspective.

A specific application of this analysis is to aid in the discovery of a malicious network (e.g. an organized crime network). Interested readers can see \cite{YD08} for a more commercial application of the same techniques. These methods are focused on the discovery of a subnetwork with particular characteristics, but they do not offer insight into the analysis of potential courses of actions that may be taken to influence the network.  For example, suppose that we consider the small five player  organization shown in Figure \ref{fig:influence}.
\begin{figure}[ht]
\centering
\includegraphics[scale=0.30]{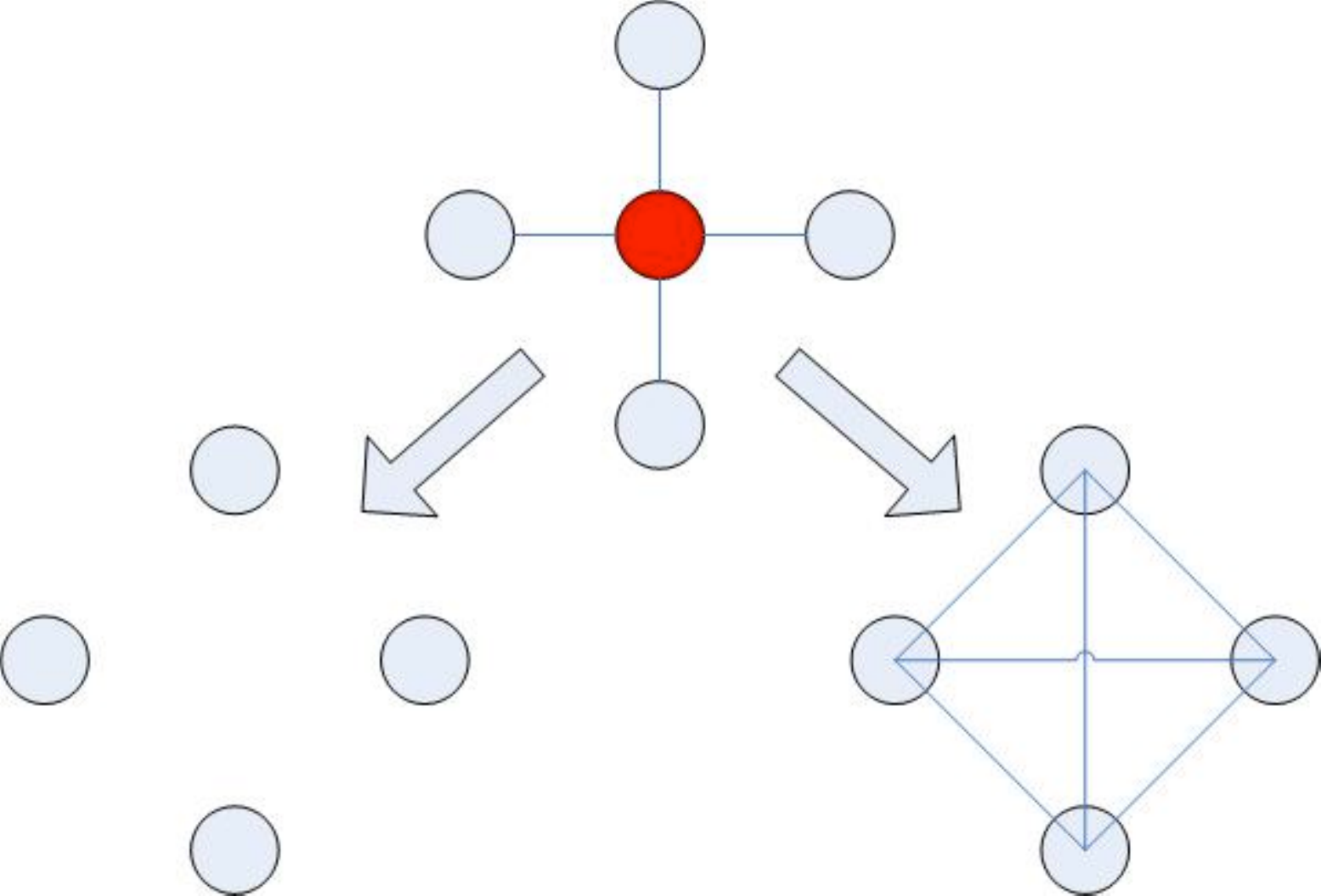}
\caption{What will happen if the red node is removed?  Will the network dissolve (left) or reform (right) }
\label{fig:influence}
\end{figure}
In this organization, the red node is the center of the organization.  It seems intuitive that the center of the organization should be attacked in some way (e.g. captured or killed).  However, it is unclear whether the organization will crumble as shown on the left by a set of disconnected nodes or if the graph will reform as a stronger network, as shown on the right.  Let us instead consider a typical firm rather than a crime organization.  If the CEO is removed for some reason (e.g. retires, is arrested, or dies), the firm may encounter a drop in the stock price or market shares as an indication of some struggling, however, it is very unlikely that it will collapse.  If the firm does collapse, it will not be quick or as a sole consequence of this event.  The firm will likely reorganize by replacing the CEO in some manner. Sometimes this benefits the firm as the new leader or reorganized organization is better than before and yet other times, the firm may encounter a steady decline under the new weaker leadership.  It is our position that the analysis of player interaction in a crime organization should use the same methodologies as those for analyzing the interactions of players in other organizational environments.

In this paper, we take the point of view that players in such a network will make strategic decisions on with whom to connect. These decisions may reflect instructions from a central authority, but it is the player who acts in his own best interest to carry these instructions out.  Using a game theoretic model of players and how they connect, will allow us to model the reformation of a network as a result of an influential act upon it.  In this respect, this approach seeks to shed light on the problem facing authorities after they discover such a malicious network. To our knowledge, no organization has attempted to use game theoretic techniques to study the network formation or reformation process to inform policy on kill or capture operations.

While a game theoretic approach allows the analysis of influence on a network, it requires knowledge of the value functions of the players involved.  It is not possible to have perfect knowledge of these value functions and we defer a discussion of this problem to future work. To be sure, application of this technique requires that player objective functions be estimated. However, such a requirement is not inconsistent with current operational requirements for identifying and targeting members of organizations.



\section{Related Literature from the Physics Community}
The Network Science community has largely dedicated its efforts to the exposition and analysis of topological properties that occur in several real-world networks; e.g., scale-freeness \cite{barabasi1999a, newman2003, dorogovtsev2002, albert2002}. Recently, there has been interest in showing that these topological properties may arise as a result of optimization, rather than some immutable physical law \cite{doyle2005}. As Doyle et al. \cite{doyle2005} point out, various networks such as communications networks and the Internet are designed by engineers with some objectives and constraints. While it is true that there is often not a single designer in control of the entire network, the network does not \emph{naturally} evolve without the influence of \emph{designers}.  In each application, the network structure must be feasible with respect to some physical constraints corresponding to the tolerances and specifications of the equipment used in the network.  For example, in a system such as the world wide web, a single web-page might have billions of connections, however it is not possible for a node to have such a degree in many other applications, such as collaboration or road networks.  Certainly the structure of the network has a significant impact on its [the network's] ability to function, its evolution, and its robustness.  However network structures often arise as a result (locally) of optimized decision making among a single agent or multiple competitive or cooperative agents, who take network structure and function into account as a part of a collection of constraints and objectives.
Recently, network formation has been modeled from a game theoretic perspective \cite{myerson1977,jackson1996,dutta1997,goyal2003}, and in \cite{LichterGriffinFriesz2011}, it was shown that there exists games that result in the formation of a stable graph with an arbitrary degree sequence.  In \cite{LGF11}, it was shown that a network with an arbitrary degree sequence may result as a pairwise stable graph with link bias in the player strategies.  In this paper, we show that we may formulate an optimization problem to measure the price of anarchy of stable graphs for this network formation game.

\section{Preliminary Notation and Definitions}
Let $N=\{1,2, \ldots n\}$ be the set of nodes (vertices) in a graph.  Following the graph formation game literature \cite{bala2000,goyal2003,goyal2006,jackson2005}, a graph is denoted as $g$ and represents a set of links (edges) (where a link is a subset of $N$ of size two). This notation is consistent with the notion of a simple graph (i.e., a graph with out multiple edges or self-loops) in the standard graph theory literature; see, e.g., \cite{Dies10}.   In this paper, it will be more convenient to denote a graph as $\mathbf{x}=\langle x_{ij} \rangle$ where:\begin{equation}
x_{ij} =x_{ji} =
\begin{cases}
1 & \text{if node $i$ is linked to node $j$}\\
0 & \text{else}
\end{cases}
\label{eqn:xdef1}
\end{equation}
We assume $x_{ii} = 0$ for all $i \in N$.
Note $\mathbf{x}$ is the symmetric adjacency matrix \cite{GR01} for a graph $g$ and we consider \textit{only} non-directed simple graphs. Denote $X$ as the set of all graphs over the node set $N$, that is, $X=\{\mathbf{x} : \mathbf{x} \in \{0,1\}^{n \times n} \}$.

Let $\eta_{i}(\mathbf{x}):X \rightarrow \mathbb{R}$ denote the degree of node $i$ in graph $\mathbf{x}$, which is computed:
\begin{equation}
\eta_{i}(\mathbf{x})=\sum_{j}x_{ij}=x_{i\cdot}\mathbf{e}
\label{eqn:eta_i_def}
\end{equation}
where $\mathbf{e}$ is an n-dimensional column vector filled with $1$'s, that is $\mathbf{e}=(1,1, \ldots 1)^T \in \mathbb{R}^{n \times 1}$.  Denote $\eta(\mathbf{x}):X \rightarrow \mathbb{R}^n$ as the degree sequence of the graph $\mathbf{x}$; i.e., $\eta(\mathbf{x})=(\eta_{1}(\mathbf{x}),\eta_{2}(\mathbf{x}),\ldots, \eta_{n}(\mathbf{x}))$, which can be computed:
\begin{equation}
\eta(\mathbf{x})=\mathbf{x} \mathbf{e}
\label{eqn:eta_def}
\end{equation}
Note, we break slightly with the notation in the literature. In (e.g.) \cite{GY05} a degree sequence is listed in descending order. We do not have this requirement as we we will be interested in conditions where certainly player's in a graph formation have a certain desired degree.

Let $[\mathbf{x}]_{\eta}$ be the equivalence class of graphs with the same degree sequence as $\mathbf{x}$.  A degree sequence $\mathbf{d}=(d_1,\dots,d_n)$ on $n$ nodes is \emph{graphical}, if there exists a graph with $n$ nodes and degree sequence $\mathbf{d}=(d_1,\dots,d_n)$. The $\ell_l$-norm between two degree sequences $\mathbf{d}=(d_1,\dots,d_n)$ and $\mathbf{k}=(k_1,\dots,k_n)$ is given as:
\begin{equation}
\|\mathbf{d}-\mathbf{k} \|_1=\sum^{n}_{i} |d_{i}-k_{i}|
\end{equation}
We define a graph $\mathbf{x}$ to be \emph{closest} in $\ell_l$ norm to degree sequence $\mathbf{d}$ if it is a graph in $X$ with a degree sequence that has the minimum $\ell_l$ norm distance to $\mathbf{d}$  of all graphs in $X$. Naturally, this graph may not be unique. Formally, $\mathbf{x}$ is a closest graph to degree sequence $\mathbf{d}$ if:
\begin{equation}
\label{eqn: defn closest}
\|\eta(\mathbf{x})-\mathbf{d}\|_{1}=\min_{\mathbf{x}' \in X} \|\eta(\mathbf{x}')-\mathbf{d}  \|_{1}
\end{equation}


\section{Network Game Notational Preliminaries}
In the network formation game, each player selects a strategy that consists of a set of preferences for with whom they would like to form a link.  The preference of each player toward each link is represented in a matrix $\mathbf{s}$ where $s_{ij}$ denotes the preference for node $i$ to link with node $j$:
\begin{equation}
s_{ij} =
\begin{cases}
1 & \text{if Player $i$ desires to link with Player $j$}\\
0 & \text{else}
\end{cases}
\label{eqn:sdef}
\end{equation}
We define $s_{ii} = 0$ to be consistent with our assumptions from the previous section. Player $i$ will select strategy $\mathbf{s}_{i \cdot}$ from the set of all possible strategies $S_{i}=\{0,1\}^{n-1}$.  In this game, each player has the ability to veto a link.  That is, the formation of a link requires participation by both players as is consistent with many social networks.  Hence each Player $i$ selects strategy $\mathbf{s}_{i \cdot}$ and all together a strategy $\mathbf{s}=[s_{i \cdot}]_{i \in N}$ is identified from $S=S_{1} \times S_{2} \times \cdots \times S_{n}$.  As a result of the strategy chosen, a simple graph $\mathbf{x}$ is formed:
\begin{equation}
x_{ij} =x_{ji} =
\begin{cases}
1 & \text{if $s_{ij} =s_{ji} =1$}\\
0 & \text{else}
\end{cases}
\label{eqn:xdef}
\end{equation}

The value of a graph $\mathbf{x}$ is the total value produced by agents in the graph; we denote the value of a graph as the function $v:X \rightarrow \mathbb{R}$ and the set of of all such value functions as $V$. An allocation rule $Y:V \times X \rightarrow \mathbb{R}^{n}$ distributes the value $v(\mathbf{x})$ among the agents in $N$.  Denote the value allocated to agent $i$ as $Y_{i}(v,\mathbf{x})$.  Since, the allocation rule must distribute the value of the network to all players, it must be \textit{balanced}; i.e., $\sum_{i} Y_{i}(v,\mathbf{x})=v(\mathbf{x})$ for all $(v,\mathbf{x}) \in V \times X$.  The allocation rule governs how the value is distributed and thus makes a significant contribution to the model.  Denote the game $\mathcal{G}= \mathcal{G}(v,Y,N)$ as the game played with value function $v$ and allocation rule $Y$ over nodes $N$. In this game, each agent (node) is a selfish profit maximizer who chooses strategy $\mathbf{s}_{i \cdot}$ to maximize there allocated payoff under allocation rule $Y$ and value function $v$.

Jackson and Wolinksy  \cite{jackson1996} suggest that  Nash Equilibrium analysis for this game may lead to inconsistencies between intuitive expected behavior and equilibrium behavior.  A Nash Equilibrium is defined as any solution such that no player can unilaterally change strategies to benefit himself.  For any potential link between two nodes $i$ and $j$, the solution with no link (i.e., $x_{ij}=0$) is an equilibrium even if both players could benefit from the link. This is because each player has link veto power and therefore, no player can \textit{unilaterally} create a link with any other player. 

In order to provide a more intuitive equilibrium concept for this model, Jackson and Wolinksy use \textit{pairwise stability} to model stable networks without the use of Nash equilibria \cite{jackson1996}.  We will first denote $\mathbf{x}+ij$ as the graph $\mathbf{x}$ with the additional link $ij$ and $\mathbf{x}-ij$ as the graph $\mathbf{x}$ without the link $ij$.
\begin{defn}
A network $\mathbf{x}$ with value function $v$ and allocation rule $Y$ is pairwise \textit{stable} if (and only if):
\begin{enumerate*}
\item for all $i,j$ if $x_{ij} = 1$, $Y_{i}(v,\mathbf{x}) \geq Y_{i}(v,\mathbf{x} - ij)$ and
\item for all $i,j$ if $x_{ij} = 0$ then, if $Y_{i}(v,\mathbf{x}+ij) > Y_{i}(v,\mathbf{x})$, then $Y_{j}(v,\mathbf{x}+ij) < Y_{j}(v,\mathbf{x})$
\end{enumerate*}
\end{defn}
Pairwise stability implies that in a stable network, for each link that exists, (1) both players must benefit from it and (2) if a link can provide benefit to both players, then it, in fact, must exist.  Jackson notes that pairwise stability may be too weak because it does not allow groups of players to add or delete links, only pairs of players \cite{jackson2003}.  Deletion of multiple links simultaneously has been considered in Belleflamme and Bloche \cite{belleflamme2004}.
Previously, in \cite{LichterGriffinFriesz2011} we extend work in \cite{jackson2008} and \cite{dutta1997}, showing that stable networks may be formed as a result of a link formation game with an arbitrary (desired) degree sequence. This work presupposes that we embed the degree sequence (implicitly or explicityly) into the objective functions of the players.

\begin{theorem}
Let $f:\mathbb{R} \rightarrow \mathbb{R}$ be a nonnegative convex function with minimum at $0$ and let $\mathbf{d} = (d_1,\dots,d_n)$ be a degree sequence. In $\mathcal{G}(v,Y,N)$, assume that Player $i$ maximizes objective $Y_{i}(\mathbf{x})=-f_{i}(\eta_{i}(\mathbf{x}))=-f(\eta_{i}(\mathbf{x})-d_{i})=-f(\sum_{j}x_{ij}-d_{i})$. 
Assume that $v$ is the balanced value function induced from the allocation rule $Y = (Y_1,\dots,Y_n)$. If $\eta^{-1}(\mathbf{d})$ is non-empty (i.e., $\mathbf{d}$ is graphical) then any graph $\mathbf{x}$ such that $\eta(\mathbf{x}) = \mathbf{d}$ is pairwise stable for the the game $\mathcal{G}(v,Y,N)$.
\label{thm:GeneralArbitraryDist}
\end{theorem}
\begin{proof}
See \cite{LichterGriffinFriesz2011}.
\end{proof}

\section{Link Bias}
The model presented in Theorem \ref{thm:GeneralArbitraryDist} requires that each player have a similar objective function and, essentially, a desired degree -- or a special number of individuals to whom he wishes to connect. While this is a good first model for many applications, and establishes the relative ease with which Network Science metrics can be incorporated into a game theoretic framework, in real-world situations each player usually has a bias toward linking with a specific set of players rather than others.  In this section, we incorporate a player's preference to link to one player over another.  The incorporation of link bias does not change the fact that given any degree sequence, there exists a game that will result in a pairwise stable graph with such a degree sequence \cite{LGF11}.

Assume each player \textit{minimizes} a cost $f_{i}:X \rightarrow \mathbb{R}$. In this framework, we assume that $Y_i(\mathbf{x}) = -f_{i}(\mathbf{x})$ and the player maximizes $Y_i(\mathbf{x})$. Consistent with Theorem  \ref{thm:GeneralArbitraryDist}, we assume the value function $v$ in $\mathcal{G}(v,Y,N)$ is defined implicitly from the $Y_i(\mathbf{x})$ ($i=1,\dots,n$) so that $v$ is balanced.  For the remainder of this section we assume a linear cost function:
\begin{equation}
f_{i}(\mathbf{x})=\sum_{j}c_{ij}x_{ij}
\end{equation}
As in the prior model a link between players $i$ and $j$ exists if and only if the two players decide to link. That is, a player may \textit{unilaterally} reject a link.
\subsection{Stability}
In this model, Player $i$ will benefit from linking with any Player $j$ whenever $c_{ij}<0$ and Player $i$ will be penalized for linking with any Player $j$ whenever $c_{ij}>0$.  We may have an unspecified behavior when $c_{ij} = c_{ji} = 0$. In this case it will neither help nor hinder either player to establish a link. To remove this possibility, we may assume link parsimony. That is, we will assume that a link is established if and only if both players benefit in some way. The stability condition becomes:
\begin{defn}
A network $\mathbf{x}$ is pairwise \textit{stable} if (and only if):
\begin{enumerate*}
\item for all $i,j$: if $x_{ij} =1 $, then $c_{ij} < 0$ and $c_{ji} < 0$
\item for all $i,j$: if $x_{ij} = 0$, then if $c_{ij} \geq 0$ then $c_{ji} \leq 0$
\end{enumerate*}
\end{defn}
We show in \cite{LGF11}, that given an arbitrary degree sequence $\mathbf{d} = (d_1,\dots,d_n)$ we may construct a cost matrix $\mathbf{c}$, whose components are the objective function coefficients $c_{ij}$, such that the resulting graph formation game has as a (pairwise) stable solution a graph with degree sequence $\mathbf{d}$. The construction of $\mathbf{c}$ is done by an optimization problem.
\begin{defn}
As in the previous section, the preference of Player $i$ to link with Player $j$ is indicated via $\mathbf{s}_{ij}$ where:
\begin{equation*}
s_{ij} = \begin{cases}
1 & \text{if player } i \text{ can benefit from link } ij\\
0 & \text{if player } i \text{ cannot benefit from link } ij
\end{cases}
\end{equation*}
\end{defn}
\begin{remark}
Specifically, $s_{ij}$ is the boolean mapping of $c_{ij}$:
\begin{equation}
s_{ij} = \begin{cases}
1 & \text{if } c_{ij}<0\\
0 & \text{if } c_{ij} \geq 0
\end{cases}
\label{defn: psi}
\end{equation}
\end{remark}

\begin{lemma}
\label{lem: stability}
A graph  $\mathbf{x} = \langle{x_{ij}}\rangle$ is stable if and only if it satisfies the following constraints:
\begin{equation}
\label{eqn: stability constraints0}
\begin{alignedat}{2}
x_{ij}&=x_{ji} &\quad& \forall \; ij\\
s_{ij}+s_{ji}-1 &\leq x_{ij} && \forall \; ij \\
x_{ij} &\leq s_{ij}  && \forall \; ij \\
x_{ij} &\leq s_{ji}  && \forall \; ij \\
x_{ij}, s_{ij} &\in \{0,1\} && \forall \; ij
\end{alignedat}
\end{equation}
\label{lem:BiasStable}
\end{lemma}
\begin{proof}
See \cite{LGF11}.
\end{proof}
\begin{remark}
Note the linearity of the constraints in both $s_{ij}$ and $x_{ij}$ suggests we could introduce additional constraints to the constraints in Expression \ref{eqn: stability constraints0} and solve for both $\mathbf{x}$ and $\mathbf{s}$. If there is a feasible solution $(\mathbf{x},\mathbf{s})$ to the constraints, then this solution can be used to generate a cost matrix $\mathbf{c}$. In fact, it can be used to generate an infinite number of cost matrices, the simplest one given by:
\begin{equation}
c_{ij} = \begin{cases}
-1 & \text{if }s_{ij} = 1\\
1 & \text{if }s_{ij} = 0
\end{cases}
\label{defn:C}
\end{equation}
\end{remark}

\begin{proposition}
\label{prop: stabilityK}
Let $\mathbf{d}=(d_{1},d_{2},\ldots, d_{n})$ be a degree sequence and consider the following constraints:
\begin{equation}
\begin{alignedat}{2}
\sum_{j \neq i} x_{ij}&=d_{i} &\quad& \forall \; i  \\
x_{ij}&=x_{ji} && \forall \; ij\\
s_{ij}+s_{ji}-1 &\leq x_{ij} && \forall \; ij \\
x_{ij} &\leq s_{ij}  && \forall \; ij \\
x_{ij} &\leq s_{ji}  && \forall \; ij \\
s_{ij} &\in \{0,1\} && \forall \; ij
\end{alignedat}
\label{eqn: stability constraints1}
\end{equation}
If $(\mathbf{x},\mathbf{s})$ is a feasible solution to the constraints in Expression \ref{eqn: stability constraints1}, $\mathbf{c}$ is a cost matrix constructed from $\mathbf{s}$ as in Expression \ref{defn:C} and $f_1,\dots,f_N : X \rightarrow \mathbb{R}$ are player cost functions with $Y_i(v,\mathbf{x}) = -f_i(\mathbf{x}) = -\sum_{j}c_{ij}x_{ij}$, then $\mathbf{x} = \langle{x_{ij}}\rangle$ is a pairwise stable solution for $\mathcal{G}(Y,v,N)$ \textbf{and} has degree sequence $\mathbf{d}=(d_{1},d_{2},\ldots, d_{n})$, where $v$ is the balanced value function defined from $Y_i$ ($i=1,\dots,n$).
\hfill\qed
\end{proposition}

\subsection{Construction of Cost Matrix via Optimization}
If there is no feasible $\mathbf{s}$ to the constraints in Expression \ref{eqn: stability constraints1}, the first constraint may be priced out and the nonlinear integer program (Expression \ref{model: optimization 1A}) may be solved to find a solution $(\mathbf{x},\mathbf{s})$ such that the resulting graph $\mathbf{x}$ is pairwise stable for $\mathcal{G}(Y,v,N)$ as defined in Proposition \ref{prop: stabilityK} and degree sequence of $\mathbf{x}$ is as close possible to $\mathbf{d}$ in the $\ell_1$ metric.

\begin{equation}
\label{model: optimization 1A}
\begin{aligned}
\min\;\; & \sum_{i} \left| \sum_{j \neq i} x_{ij}-d_{i} \right|\\
s.t.\; & x_{ij} - x_{ji} = 0 &\quad& \forall\; i < j\\
& s_{ij}+s_{ji}-1 \leq x_{ij} && \forall \; ij\\
&x_{ij} \leq s_{ij}  && \forall \; ij \\
&x_{ij} \leq s_{ji}  && \forall \; ij \\
& x_{ij},s_{ij} \in \{0,1\} && \forall\; i,j
\end{aligned}
\end{equation}
Problem \ref{model: optimization 1A} may be reformulated as Problem \ref{model: optimization 1B}, which has linear objective function and nonlinear constraints.
\begin{equation}
\label{model: optimization 1B}
\begin{aligned}
\min\;\; & \sum_{i} e_{i}\\
s.t.\; &  \left| \sum_{j \neq i} x_{ij}-d_{i} \right| \leq e_{i}\\
&x_{ij} - x_{ji} = 0 &\quad& \forall\; i < j\\
& s_{ij}+s_{ji}-1 \leq x_{ij} && \forall \; ij\\
&x_{ij} \leq s_{ij}  && \forall \; ij \\
&x_{ij} \leq s_{ji}  && \forall \; ij \\
& x_{ij},s_{ij} \in \{0,1\} && \forall\; i,j\\
\end{aligned}
\end{equation}
This problem can be reformulated as a purely linear integer programming problem (Problem \ref{model: optimization 1C}).
\begin{equation}
\label{model: optimization 1C}
\begin{aligned}
\min\;\;  &\sum_{i} e_{i} &\\
s.t.\;   &\sum_{j \neq i} x_{ij}-d_{i}  \leq e_{i} &\quad& \forall\; i\\
 - &\sum_{j \neq i} x_{ij}+d_{i}  \leq e_{i} && \forall\; i\\
 & s_{ij}+s_{ji}-1 \leq x_{ij} && \forall \; ij\\
 &x_{ij} \leq s_{ij}  && \forall \; ij \\
&x_{ij} \leq s_{ji}  && \forall \; ij \\
&x_{ij} - x_{ji}= 0 && \forall \; i < j\\
& x_{ij},s_{ij} \in \{0,1\} && \forall \; i,j\\
\end{aligned}
\end{equation}

\begin{remark} The following theorem is an immediate result of the equivalence of Problem \ref{model: optimization 1C} and Problem \ref{model: optimization 1A} and Proposition \ref{prop: stabilityK}.
\end{remark}

\begin{theorem}

Suppose $\mathbf{d}=(d_{1},d_{2},\ldots, d_{n})$ is a degree sequence, $(\mathbf{x},\mathbf{s})$ is an optimal solution to Problem \ref{model: optimization 1C}, $\mathbf{c}$ is a cost matrix constructed from $\mathbf{s}$ as in Expression \ref{defn:C} and $f_1,\dots,f_N : X \rightarrow \mathbb{R}$ are player cost functions with $Y_i(v,\mathbf{x}) = -f_i(\mathbf{x}) = -\sum_{j}c_{ij}x_{ij}$, then $\mathbf{x} = \langle{x_{ij}}\rangle$ is a pairwise stable solution for $\mathcal{G}(Y,v,N)$ \textbf{and} has degree sequence as close as possible to $\mathbf{d}$ in the $\ell_1$ metric, where $v$ is the balanced value function defined from $Y_i$ ($i=1,\dots,n$).
\label{thm:GeneralArbitraryBiasDist}
\end{theorem}
\begin{proof}
See \cite{LGF11}.
\end{proof}

\begin{remark} Theorem \ref{thm:GeneralArbitraryBiasDist} shows how to construct a link biased game that encapsulates a specific degree sequence $\mathbf{d}$. The resulting game (or class of games) hides the degree distribution in the objective functions of the players in a non-obvious way, illustrating how arbitrary degree distributions can be hidden within the objective functions of individual players. This also suggests the need to investigate general value function inference from data, which we discuss in our future work.
\end{remark}

\section{Price of Anarchy}
When networks form as a result of selfish competition among nodes, the resulting stable network may not, in fact, be system optimal. It is possible that a stable configuration is achieved in which each node does worse than if a central planner had optimized the system. In this circumstance, we would like to measure the collective penalization due to decentralized control.  This collective price has been called the \textit{price of anarchy}, which is measured as the ratio between the worst equilibrium and the centralized solution.

The price of anarchy \cite{KP99} is a method to measure the inefficiency of equilibrium, but the development of the analyisis of the inefficiency of equilibriums predates the price of anarchy \cite{dubey1986}.  There already exists multiple variations of the price of anarchy (e.g. Pure Price of Anarchy, Mixed Price of Anarchy), for the different types of equilibriums that exist.  In this paper, we have defined pairwise stability rather than using the typical Nash Equilibrium, so this foreshadows the presentation of an analogue to the price of anarchy in the sequel.

The price of anarchy has been used to measure the inefficiency in congestion networks \cite{Ros73, GK05}.  In these games, each user of the network has a source and destination and they must pay a cost to travel from their source to their destination.  The latency or cost of each link in the network increases with congestion, that is as the number of players that use it increases.  Economically, we may think of \emph{capacity} on the network (specifically between an origin destination pair) as a commodity being supplied.  As more users want that commodity, the price increases.  Each user will selfishly minimize his own cost via route selection.  The inefficiency of the system is calculated by measuring the total cost of the system when users act selfishly versus the total cost when users act in a coordinated manner by a centralized solution.

The price of anarchy is calculated as:
\begin{equation}
\label{eqn: POA1}
\rho=\frac{\max_{\mathbf{x} \in X}{\sum_{i}{Y_{i}(\mathbf{x})}}}{\min_{\mathbf{x} \in E} \sum_{i}{Y_{i}(\mathbf{x})}}=\frac{\max_{\mathbf{x} \in E}{\sum_{i}{f_{i}(\mathbf{x})}}}{\min_{ \mathbf{x} \in X} \sum_{i}{f_{i}(\mathbf{x})}}
\end{equation}
where the set $E \subset X$ is defined as the subset of equilibrium solutions.  In this system, this ratio is a useful measurement because $\min_{\mathbf{x} \in X} \sum_{i}{f_{i}(\mathbf{x})}>0$.  In models such as the traffic congestion model, there is an inherent minimum cost to traverse a link due to the fact that each link has a certain minimal traversal time.  The congestion from other players and lack of capacity causes an additional cost, which is measured as a ratio of the physical cost.  However, this property is simply not true of all games where players act selfishly.  In models of network formation, there is not necessarily an inherent cost to the system.  It is possible that due to the desires and constraints of the players involved, there is some cost that must be absorbed by some set of players, but this is due to the inconsistency of the players' desires, resources, and constraints.  In the case where there is no such inconsistency and it is possible for all players to satisfy their desires, there is an optimal cost of zero.  In these circumstances, for these applications, it does not make sense to examine the ratio as the denominator may be zero.  Hence, we calculate the \emph{price of anarchy} as:
\begin{equation}
\label{eqn: POA2}
\rho=\max_{\mathbf{x} \in X}{\sum_{i}{Y_{i}(\mathbf{x})}}-\min_{\mathbf{x} \in E} \sum_{i}{Y_{i}(\mathbf{x})}=\max_{\mathbf{x} \in E}{\sum_{i}{f_{i}(\mathbf{x})}}-\min_{ \mathbf{x} \in X} \sum_{i}{f_{i}(\mathbf{x})}
\end{equation}
In this measurement, we measure the total additional cost of the worst selfish equilibrium over the best coordinated solution and denote this as the price of anarchy.

The price of anarchy is a measure of the collective unhappiness of players due to selfishness.  It is calculated as a function of the objective function of each player. Unless otherwise stated, we will return to the assumptions made in Section 4 that $Y_{i}(\mathbf{x})=-f_{i}(\eta_{i}(\mathbf{x}))=-f(\eta_{i}(\mathbf{x})-d_{i})$ where $f$ is convex with a minimum at $0$. The shape of the function $f$ will have a considerable influence on which graph is the \emph{worst} stable.  An infinitely steep function $f$ could have an infinitely large objective value for the worst equilibrium.

For computational ease and to obtain closed form theoretical results, we will consider $f_{i}(\eta_{i}(\mathbf{x}))=f(\eta_{i}(\mathbf{x})-d_{i})=|\eta_{i}(\mathbf{x})-d_{i}|$.  This function provides additional insight into the game because, for the best and worst graph, it measures the $\ell_1$ distance between the degree sequence of the graph generated through selfish decision making and a graph with the degree sequence closest in the $\ell_1$ metric to the target degree sequence $\mathbf{d}$.

\subsection{The Worst Stable Graph}
\subsubsection{Link Bias Game}
\begin{proposition} Assume $\mathbf{c}$ is a cost matrix in the link bias game $\mathcal{G}(Y,v,N)$ where $Y_i = -\sum_{j}c_{ij}x_{ij}$ and $v$ is the balanced value function induced by $Y_i$ ($i = 1,\dots,n)$. Let $s_{ij}$ be defined by Expression \ref{defn: psi}. If $\mathbf{x}$ is a solution to the integer programming problem:
\begin{equation}
\begin{aligned}
\max\;\; & \sum_{i}\sum_{j}c_{ij}x_{ij}\\
s.t. \;\; & x_{ij}=x_{ji} &\quad& \forall \; ij\\
& s_{ij}+s_{ji}-1 \leq x_{ij} && \forall \; ij \\
& x_{ij} \leq s_{ij}  && \forall \; ij \\
& x_{ij} \leq s_{ji}  && \forall \; ij \\
& x_{ij} \in \{0,1\} && \forall \; ij
\end{aligned}
\label{eqn:LinkBiasWorstGraph}
\end{equation}
then $\mathbf{x}$ is a stable graph with the minimum net payoff.
\end{proposition}
\begin{proof} This is clear from the objective function and Lemma \ref{lem:BiasStable}.
\end{proof}
\subsubsection{Degree Sequence Game}
Next, we consider the model where Player $i$ has an allocation function $Y_{i}(v,\mathbf{x})=-f_{i}(\eta_{i}(\mathbf{x}))$ where $f_{i}:\mathbb{R} \rightarrow \mathbb{R}$ is convex with a minimum at $d_{i}$.  In this section, we define an integer program to find the stable graph with the worst total allocation for the players involved in the game.  The feasible region of this integer program will be the set of stable graphs.

\begin{remark} Define:
\begin{equation}
r_{ij}(\mathbf{x}) = \begin{cases}
1 & \text{if $Y_{i}(v,\mathbf{x}) \geq Y_{i}(v,\mathbf{x} - ij)$}\\
0 & \text{else}
\end{cases}
\end{equation}
and
\begin{equation}
p_{ij}(\mathbf{x}) = \begin{cases}
1 & \text{if $Y_{i}(v,\mathbf{x}+ij) > Y_{i}(v,\mathbf{x})$}\\
0 & \text{else}
\end{cases}
\end{equation}
and
\begin{equation}
q_{ij}(\mathbf{x}) = \begin{cases}
1 & \text{if $Y_{j}(v,\mathbf{x}+ij) < Y_{j}(v,\mathbf{x})$}\\
0 & \text{else}
\end{cases}
\end{equation}
Since stability is simply a propositional statement with propositions $r_{ij}(\mathbf{x})$, $p_{ij}(\mathbf{x})$ and $q_{ij}(\mathbf{x})$, it can be shown  that the following nonlinear integer programming problem will produce the stable graph with the minimal net payoff \cite{GTR11}:
\begin{equation}
\begin{aligned}
\min\;\; & \sum_{i} Y_i(v,\mathbf{x})\\
&x_{ij} = x_{ji} &\quad& \forall\; i < j\\
&(1-x_{ij}) + r_{ij}(\mathbf{x}) \geq 1 && \forall\; i,j \\
&x_{ij} + (1-p_{ij}(\mathbf{x})) + q_{ij}(\mathbf{x}) \geq 1 && \forall\; i,j\\
&x_{ij} \in \{0,1\} && \forall\; i,j
\end{aligned}
\end{equation}
However, the integer programming problem in its raw form is highly non-linear and therefore not efficient for computing the worst stable graph that can result from an arbitrary $Y_i(v,\mathbf{x})$. \end{remark}

\begin{lemma}
A graph $\mathbf{x} = \langle{x_{ij}}\rangle$ with value function $v$ and allocation rule $Y_{i}(v,\mathbf{x})=-f_{i}(\eta_{i}(\mathbf{x}))$ where $f_{i}$ is convex and has a minimum at $d_{i}$ is pairwise \textit{stable} if and only if:
\begin{enumerate*}
\item for all $i$, $\sum_{j\neq i} x_{ij} \leq d_i$
\item for all $i,j\neq i$, if $x_{ij} = 1$, then $\sum_{l \neq i}x_{il} \leq d_{i}$ and $\sum_{l \neq j}x_{lj} \leq d_{j}$
\item for all $i,j\neq i$, $\sum_{l \neq i}x_{il} < d_{i}$ and $\sum_{l \neq j}x_{lj} < d_{j} \implies x_{ij} = 1$
\end{enumerate*}
\label{lem:Setup}
\end{lemma}
\begin{proof} Suppose $\mathbf{x}$ is pairwise stable. If for any Player $i$, $\sum_{j\neq i} x_{ij} > d_i$, then Player $i$ could unilaterally drop one link and obtain a more favorable payoff. Thus, it is clear $\sum_{j\neq i} x_{ij} \leq d_i$ since we assumed Condition 1 of pairwise stability that  for all $i, j$, if $x_{ij} = 1$, then $Y_{i}(v,\mathbf{x}) \geq Y_{i}(v,\mathbf{x} - ij)$. Thus Constraint 1 must be true. Constraint 2 follows from this argument as well. If we assume for all $i,j$ if $x_{ij} = 0$ then, if $Y_{i}(v,\mathbf{x}+ij) > Y_{i}(v,\mathbf{x})$, then $Y_{j}(v,\mathbf{x}+ij) < Y_{j}(v,\mathbf{x})$ then this is equivalent to assuming if $x_{ij} = 0$ then if $\sum_{k\neq i}x_{ik} < d_i$ then $\sum_{k \neq j} x_{jk} = d_j$. If we take the contrapositive of Constraint 3, then we obtain: if $x_{ij} = 0$ then
$\sum_{i \neq j} x_{ik} = d_i$ or $\sum_{k \neq j} x_{jk} = d_j$. Using a simple logical equivalence, we may rewrite this expression as: if $x_{ij} = 0$ then if $\sum_{i \neq j} x_{ik} \neq d_i$ then $\sum_{k \neq j} x_{jk} = d_j$. Since we've proved Constraint 1 must hold, $\sum_{i \neq j} x_{ik} \neq d_i$ is equivalent to $\sum_{i \neq j} x_{ik} < d_i$. Thus, Constraint 3 follows from Condition 2 of pairwise stability.

Conversely, suppose these three conditions hold. The logical equivalence between Constraints 1 and Constraint 3 and Condition 2 of pairwise stability has already been established in the forgoing argument. By Constraint 2, we know that if $x_{ij} = 1$, then each Player $i$ has degree at most $d_i$, the value that maximizes his payoff function. Thus, setting $x_{ij} = 0$ (effectively constructing $\mathbf{x} - ij$ would yield a lower payoff for both Player $i$ and Player $j$. Thus, Condition 1 of pairwise stability is established.
\end{proof}

\begin{defn}
Let $\mathbf{u}$ be a vector of slack variables on the degrees of the nodes in the game. Then:
\begin{equation*}
u_{i} = \begin{cases}
d_{i}-\sum_{l \neq i}x_{il} & \text{if } \sum_{l \neq i}x_{il} < d_{i}\\
0 & \text{else } (i.e. \sum_{l \neq i}x_{il}=d_{i})
\end{cases}
\end{equation*}
Similarly, let $\mathbf{z}$ to be the binary mapping of the vector $\mathbf{u}$ with:
\begin{equation*}
z_{i} = \begin{cases}
1 & \text{if } u_{i}>0 \text{  }(i.e. \sum_{l \neq i}x_{il} < d_{i})\\
0 & \text{if } u_{i}=0 \text{  }(i.e. \sum_{l \neq i}x_{il}=d_{i})
\end{cases}
\end{equation*}
\label{defn:UV}
\end{defn}

\begin{remark} Note we will not consider the case when $\sum_{l \neq i}x_{il} > d_{i}$ because each player has the unilateral power to veto any connection.
\end{remark}

\begin{lemma}
\label{lem: stability1} Let $Y_i(v,\mathbf{x})$ be as in the statement of Lemma \ref{lem:Setup}. Let $\mathbf{d}$ be a degree sequence and let $\mathbf{x} = \langle{x_{ij}}\rangle$ with vectors $\mathbf{u}$,$\mathbf{z}$ derived from Definition \ref{defn:UV}. Then $\mathbf{x}$ is stable if and only if the following constraints are satisfied:
\begin{alignat}{2}
\sum_{j \neq i}x_{ij}+u_{i}&=d_{i} &\;\;&\forall\; i  \label{eq: constraint 1}\\
z_{i}+z_{j}-1 &\leq x_{ij} && \forall\; i, j \neq i \label{eq: constraint 2}\\
z_{i}&\leq u_{i}&& \forall\; i \label{eq: constraint 3}\\
u_{i} &\leq d_{i} z_{i}&& \forall\; i  \label{eq: constraint 4}\\
x_{ij} &= x_{ji} && \forall\; i,j \neq i\\
u_{i}& \geq 0&& \forall\; i  \label{eq: constraint 5}\\
z_{i} &\in \{0,1\} && \forall\; i \label{eq: constraint 6}\\
x_{ij} & \in \{0,1\} && \forall\; i, j\neq i
\end{alignat}
\label{lem:StabilityIP}
\end{lemma}
\begin{proof} Suppose that $\mathbf{x}$ is stable and $\mathbf{u}$ and $\mathbf{z}$ are defined appropriately. Clearly, $x_{ij} = x_{ji}$ holds. By Lemma \ref{lem:Setup} we must have $\sum_{l \neq j}x_{lj} \leq d_{j}$ and so $u_i \geq 0$ and Constraint \ref{eq: constraint 1} holds by definition. Further it is clear $0 \leq u_i \leq d_i$ for all $i$. By Definition \ref{defn:UV}, if $u_i = 0$, then $z_i = 0$ and thus $z_i \leq u_i$ and $u_i \leq d_iz_i$. The fact that $u_i \in \mathbb{Z}_+$ is ensured by the integrality of $\mathbf{x}$ so if $u_i > 0$ then $u_i \geq 1$ thus $u_i \geq z_i = 1$ and as we observed $u_i \leq d_i = d_iz_i$.

Conversely, suppose the constraints hold. That $\mathbf{x}$ is a graph is clear. Trivially, $\sum_{l \neq j}x_{lj} \leq d_{j}$. Suppose that $u_i, u_j > 0$. Then necessarily $z_i,z_j = 1$ since Constraint \ref{eq: constraint 4} must hold. Thus, by Constraint \ref{eq: constraint 2} $x_{ij} = 1$ and we have established the first and third constraint of Lemma \ref{lem:Setup}. Conversely, suppose that $x_{ij} = 1$. Then $z_i$ and $z_j$ may take any value to satisfy Constraint \ref{eq: constraint 2} and the second constraint of Lemma \ref{lem:Setup} must hold. This completes the proof.
\end{proof}

\begin{theorem} Let $Y_i(v,\mathbf{x})$ be as in the statement of Lemma \ref{lem:Setup}. Let $\mathbf{d}$ be a degree sequence. The solution to the following integer programming problem yields a graph $\mathbf{x}$ that is stable and has the worst net payoff function of any stable graph.
\begin{equation}
\label{model: WorstGraphIP}
\begin{alignedat}{3}
\max\;\; &\sum_{i} u_{i}\\
s.t.\;\;& \sum_{j \neq i}x_{ij}+u_{i}=d_{i} &\;\;&\forall\; i \\
&z_{i}+z_{j}-1 \leq x_{ij} && \forall\; i, j \neq i \\
&z_{i}\leq u_{i}&& \forall\; i \\
&u_{i} \leq d_{i} z_{i}&& \forall\; i  \\
&x_{ij} = x_{ji} && \forall\; i,j \neq i\\
&u_{i} \geq 0&& \forall\; i  \\
&z_{i} \in \{0,1\} && \forall\; i \\
&x_{ij}  \in \{0,1\} && \forall\; i, j\neq i
\end{alignedat}
\end{equation}
\end{theorem}
\begin{proof}
The statement follows at once from Lemma \ref{lem:StabilityIP} and the assumptions on $Y_i(v,\mathbf{x})$ made in Lemma \ref{lem:Setup}.
\end{proof}

\begin{example} Consider a simple example with 10 players who each have convex cost functions with minima at $5$. That is, each player would prefer to link with no more and no less than $5$ other players. Thus, the ideal graph solution is one in which each player resides in a 5-regular graph. This setup yields an instantiation of the integer programming problem in Expression \ref{model: WorstGraphIP}:
\begin{displaymath}
\begin{alignedat}{3}
\max\;\; &\sum_{i} u_{i}\\
s.t.\;\;& \sum_{j \neq i}x_{ij}+u_{i}=5 &\;\;&\forall\; i \\
&z_{i}+z_{j}-1 \leq x_{ij} && \forall\; i, j \neq i \\
&z_{i}\leq u_{i}&& \forall\; i \\
&u_{i} \leq 5 z_{i}&& \forall\; i  \\
&x_{ij} = x_{ji} && \forall\; i,j \neq i\\
&u_{i} \geq 0&& \forall\; i  \\
&z_{i} \in \{0,1\} && \forall\; i \\
&x_{ij}  \in \{0,1\} && \forall\; i, j\neq i
\end{alignedat}
\end{displaymath}
with solution visualized in Figure \ref{fig:StableBadGraph}:
\begin{figure}[htbp]
\centering
\includegraphics[scale=0.3]{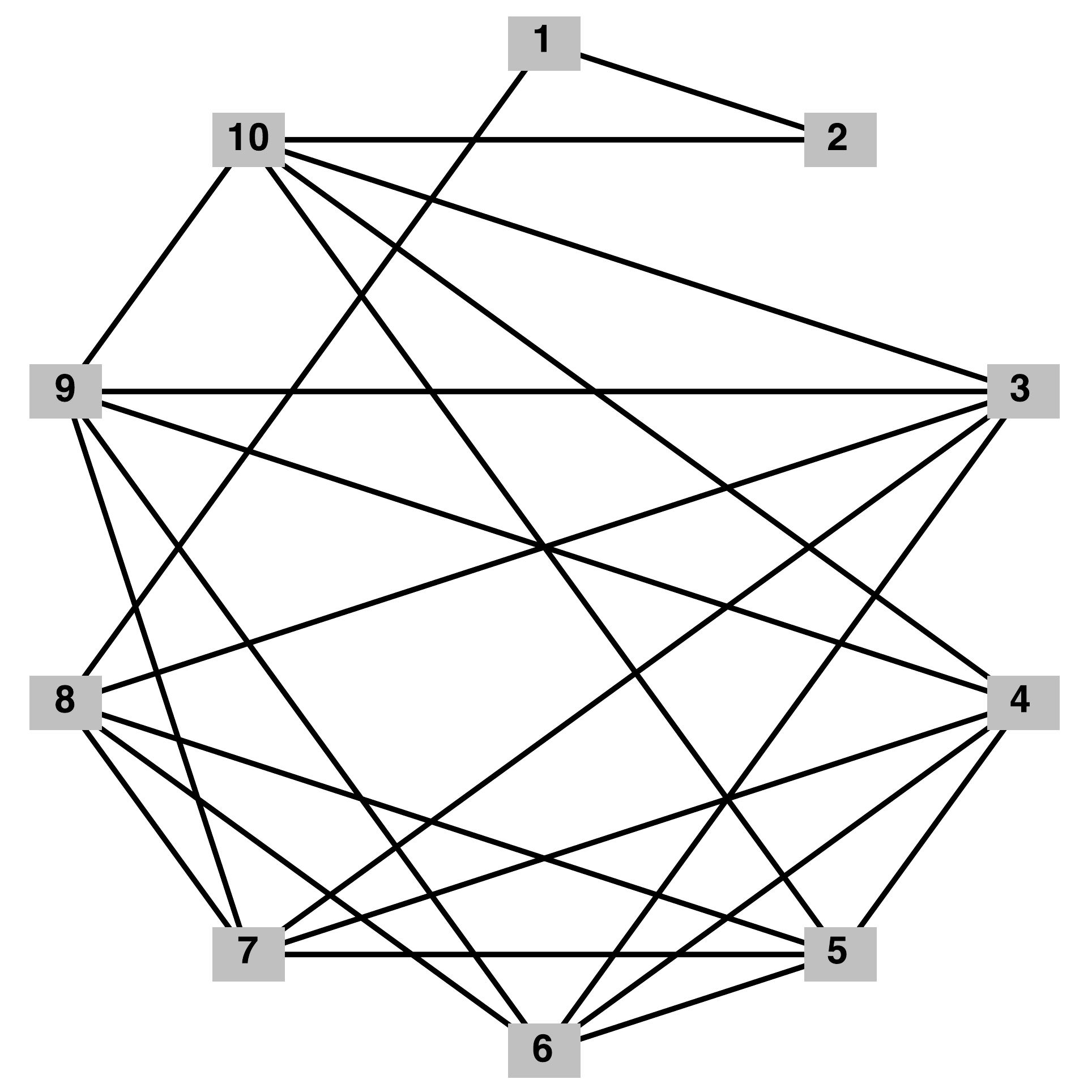}
\caption{The worst solution to the problem in which 10 players each desired to link with 5 players. The objective function in this case is 6.}
\label{fig:StableBadGraph}
\end{figure}
If we assume the objective function of each player is $f_i(\eta(\mathbf{x})) = |\eta_i(x) - 5|$, then the worst case net payoff is $6$. In this solution, we see Players 1 and 2 both have degree 2, each missing 3 connections.
\label{ex:Regular}
\end{example}

\subsection{The Best Graph}
\subsubsection{Link Bias Game}
\begin{proposition}
Assume $\mathbf{c}$ is a cost matrix in the link bias game $\mathcal{G}(Y,v,N)$ where $Y_i = -\sum_{j}c_{ij}x_{ij}$ and $v$ is the balanced value function induced by $Y_i$ ($i = 1,\dots,n)$. Let $s_{ij}$ be defined by Expression \ref{defn: psi}. If $\mathbf{x}$ is a solution to the integer programming problem:
\begin{equation}
\begin{aligned}
\min\;\; & \sum_{i}\sum_{j}c_{ij}x_{ij}\\
s.t. \;\; & x_{ij}=x_{ji} &\quad& \forall \; ij\\
& x_{ij} \in \{0,1\} && \forall \; ij
\end{aligned}
\label{eqn:LinkBiasBestGraph}
\end{equation}
then $\mathbf{x}$ is a (potentially unstable) graph with the maximum net payoff.\hfill\qed
\end{proposition}

\subsubsection{Degree Sequence Game}
In this section, we define an integer program to find the graph with the best total allocation for the players involved in the game.  Note that we are not necessarily looking for a stable graph, so the feasible region is the set of all graphs.  This will provide a baseline to evaluate the worst price that may be paid for selfish competition (e.g. the Price of Anarchy).

\begin{remark} In the general case, the following nonlinear integer programming problem will yield the graph that provides the largest net payoff:
\begin{equation}
\begin{aligned}
\max\;\; & \sum_{i} Y_i(v,\mathbf{x})\\
&x_{ij} = x_{ji} &\quad& \forall\; i < j\\
&x_{ij} \in \{0,1\} && \forall\; i,j
\end{aligned}
\end{equation}
Again, the integer programming problem in its raw form may be highly non-linear and therefore not efficient for computing the graph of interest for any possible $Y_i(v,\mathbf{x})$.
\end{remark}

Previously, there has been work on generating graphs with an arbitrary graphical degree sequence \cite{britton2006,milo2003,del2010,mihail2002}.  However this literature is mainly concerned with the algorithms to generate a graph for a graphical degree sequence.  Here we seek to find the closest graph to any degree sequence (graphical or not) and we use an optimization formulation to do this. That is, we focus on the problem in which $Y_i(v,\mathbf{x}) = -|\eta_i(x) - d_i|$. For our specific function, we formulate a math program by defining the feasible region as the set of all graphs and then minimizing the sum of the player's cost due to penalization for acquiring a degree different than desired.

\begin{equation}
\label{model: POA_best_1}
\begin{aligned}
\min\;\; & \sum_{i} \left| \sum_{j \neq i} x_{ij}-d_{i} \right|\\
s.t.\; & x_{ij} - x_{ji} = 0 \quad \forall \; i < j\\
& x_{ij} \in \{0,1\} \quad \forall\; i,j\\
\end{aligned}
\end{equation}
This nonlinear integer programming problem is easily reformulated to a linear integer programming problem:
\begin{equation}
\label{model: optimization 1B3}
\begin{aligned}
\min\;\; & \sum_{i} e_{i} \\
s.t.\;   &\sum_{j \neq i} x_{ij}-d_{i}  \leq e_{i} &\quad& \forall \;i \\
- &\sum_{j \neq i} x_{ij}+d_{i}  \leq e_{i} && \forall i \; \\
&x_{ij} - x_{ji}= 0 && \forall\; i < j\\
& x_{ij} \in \{0,1\} && \forall\; i,j\\
\end{aligned}
\end{equation}
This integer program minimizes the distance between the arbitrary degree sequence $\mathbf{d}=\{d_1,\dots,d_n\}$ and the degree sequence of a graph in the feasible region.
\begin{theorem}
The graph generated by an optimal solution to the integer program:
\begin{equation}
\label{model: optimization 1B2}
\begin{aligned}
\min\;\; & \sum_{i} e_{i} \\
s.t.\;   &\sum_{j \neq i} x_{ij}-d_{i}  \leq e_{i} &\quad& \forall\; i \\
- &\sum_{j \neq i} x_{ij}+d_{i}  \leq e_{i} && \forall\; i \\
&x_{ij} - x_{ji}= 0 && \forall\; i < j\\
& x_{ij} \in \{0,1\} && \forall\; i,j\\
\end{aligned}
\end{equation}
is a closest graph under the $\ell_{1}$-norm to a graph with degree sequence $\mathbf{d}=\{d_1,\dots,d_n\}$. \hfill\qed
\end{theorem}
\begin{example}
Now, the price of anarchy is simply the difference of the objective function value from \emph{the worst graph} (Problem (\ref{model: WorstGraphIP})) to \emph{the best graph} (Problem (\ref{model: optimization 1B2})). Continuing Example \ref{ex:Regular}, the degree sequence in question is graphical. Thus, a globally optimal solution is one in which each player is adjacent to 5 other players. This is shown in Figure \ref{fig:StableBestGraph}.
\begin{figure}[htbp]
\centering
\includegraphics[scale=0.3]{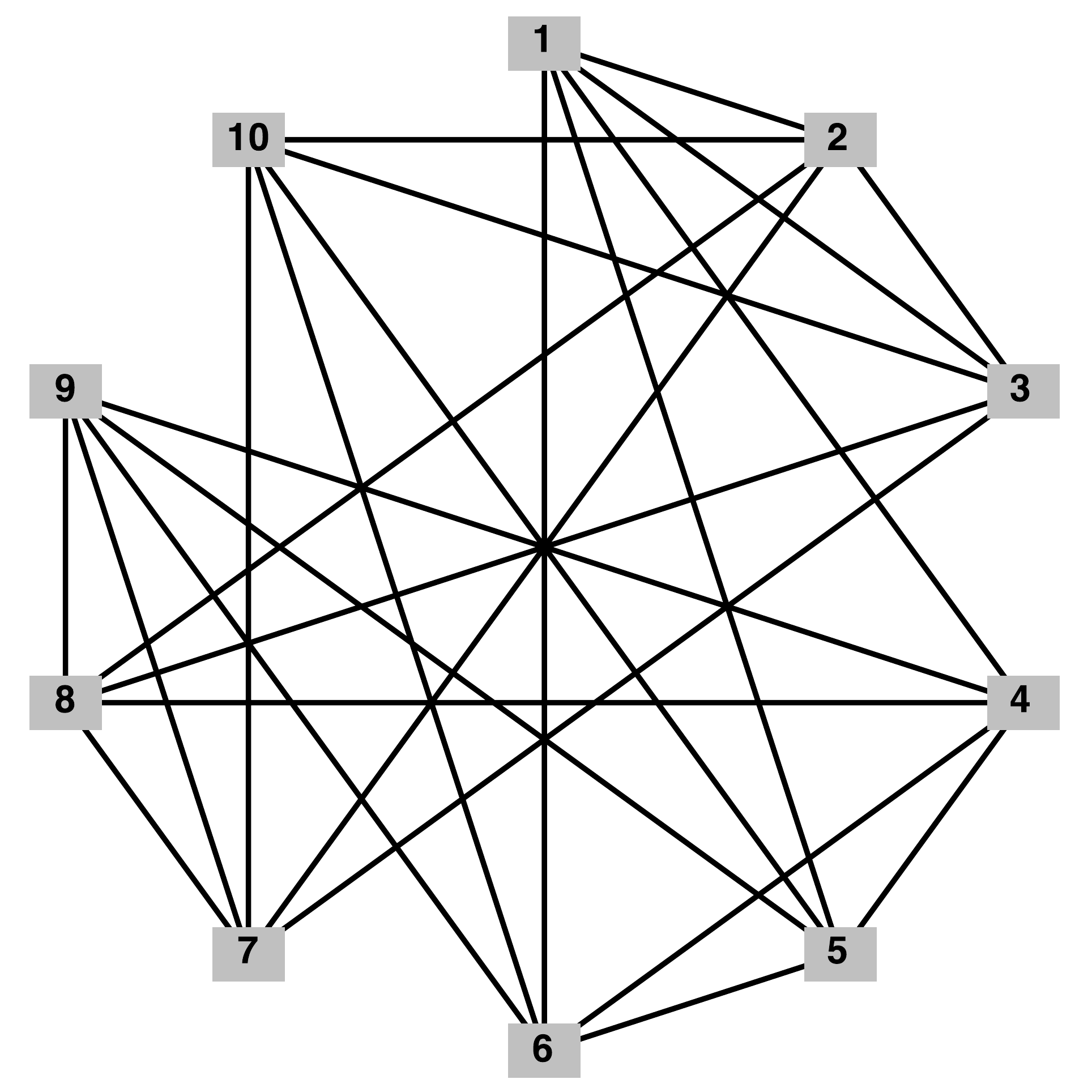}
\caption{A best solution to the problem in which 10 players each desired to link with 5 players. The objective function in this case is 0 and moreover, this graph is stable.}
\label{fig:StableBestGraph}
\end{figure}
We note that this graph is not only a globally optimal solution, it is also a pairwise stable graph. As before, we are assuming that $f_i(\eta(\mathbf{x})) = |\eta_i(x) - 5|$. Since the objective function value is $0$, it is easy to see that that price of anarchy as defined in Equation \ref{eqn: POA2} is 5.
\end{example}

\subsection{Complete Example}
Using a link bias game, we can illustrate a complete example of the process by which a modeler might make use of these techniques. Suppose that after studying a 10 player decentralized network, a cost matrix $\mathbf{c}$ is constructed to identify link bias between players. This cost matrix might be:
\begin{equation}
\mathbf{c} = \left[ \begin {array}{cccccccccc} 0&-85&-29&13&-25&-94&-19&-97&10&10
\\ \noalign{\medskip}75&0&9&32&78&27&-55&-38&-44&-61
\\ \noalign{\medskip}-85&19&0&48&23&18&71&-36&26&-26
\\ \noalign{\medskip}-19&25&35&0&-67&18&-50&-69&-3&-20
\\ \noalign{\medskip}57&17&80&51&0&63&-17&69&-62&-78
\\ \noalign{\medskip}83&81&20&20&-81&0&35&-15&-83&-4
\\ \noalign{\medskip}-45&89&39&-46&-36&-51&0&2&9&5
\\ \noalign{\medskip}68&92&-35&35&-88&51&-86&0&88&-91
\\ \noalign{\medskip}58&-2&26&-54&91&38&50&99&0&-44
\\ \noalign{\medskip}-43&-46&-74&-17&-62&-38&-94&-59&63&0\end {array}
 \right]
\end{equation}
Using this information, the worst net payoff stable graph can be identified from Problem \ref{eqn:LinkBiasWorstGraph}. The value to the organization under this strategy is 1077 units of reward. If the group were organized centrally, the value to the group would be 1487 units of reward, computed from  Problem \ref{eqn:LinkBiasBestGraph}. The resulting graphs are shown in Figure \ref{fig:InitialGraph}.
\begin{figure}[htpb]
\centering
\subfigure[Worst Graph]{\includegraphics[scale=0.3]{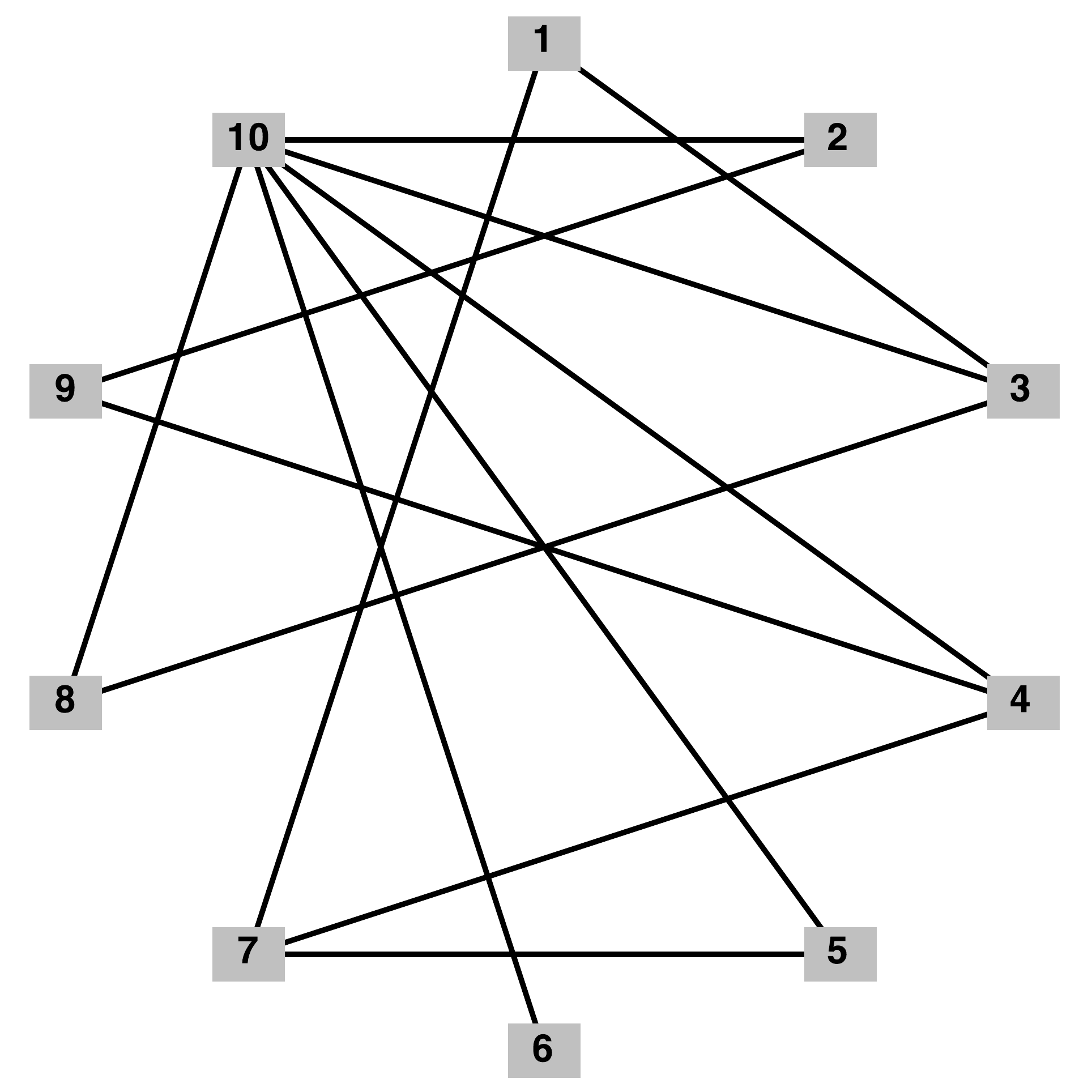}}
\subfigure[Best Graph]{\includegraphics[scale=0.3]{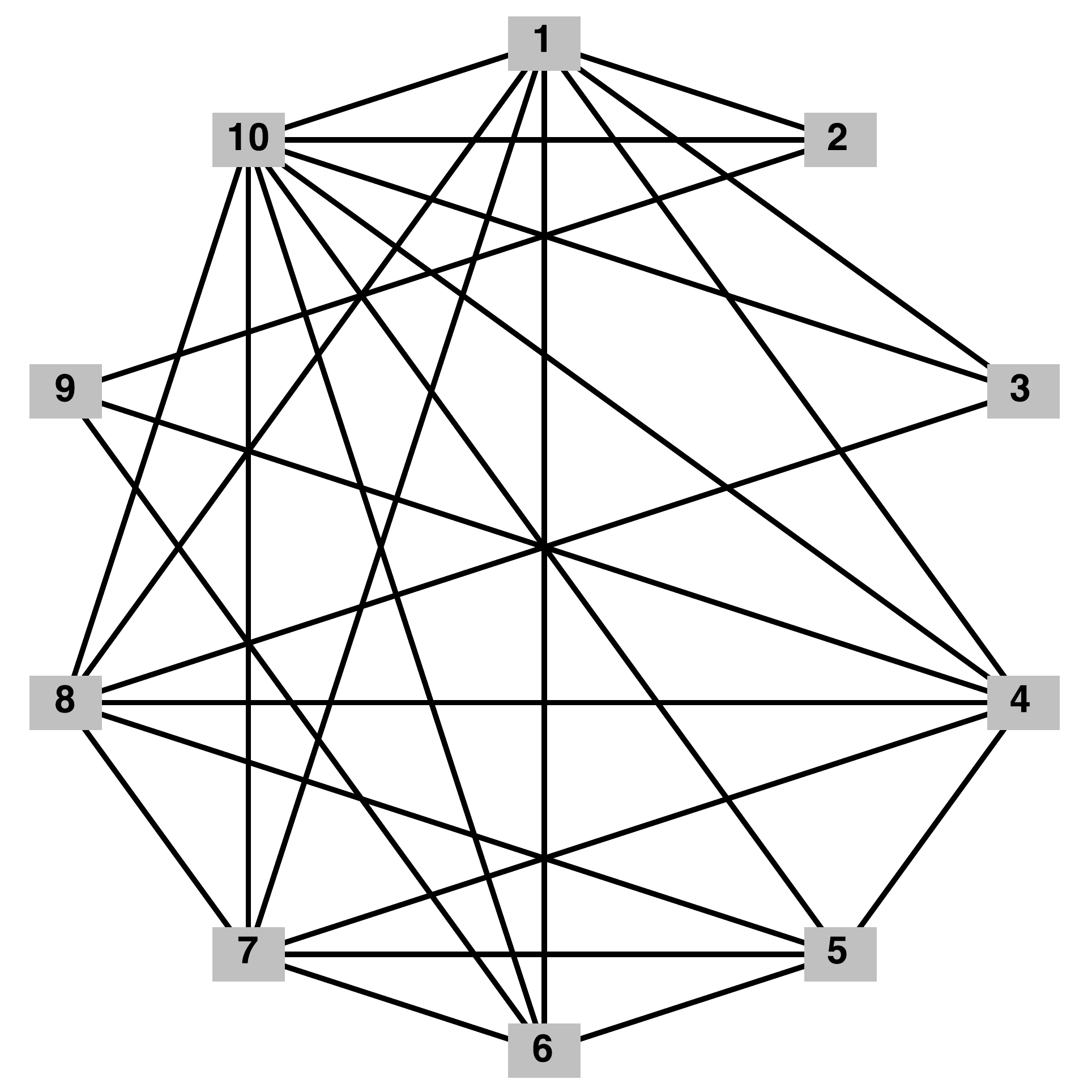}}
\caption{The worst stable graph and best centrally coordinated graph using the cost matrix $\mathbf{c}$.}
\label{fig:InitialGraph}
\end{figure}
Since the two community payoffs are positive, we can analyze the price of anarchy using the traditional ratio method and we see that decentralized play leads to a approximately $\sim 72\%$ of the payoff that would come from centralized coordination. Naturally, if the player's true network resembled the centrally coordinated graph rather than the uncoordinated (stable) graph, we might suspect this network was \textit{not} selfishly coordinating and our payoff matrix was incorrect.

Suppose now we isolate a vertex and target it for kill or capture. Without loss of generality, we have identified that it is possible to kill or capture either Vertex 10 or Vertex 1. In an ordinary network analysis, Vertex 10 is clearly a high priority target since it has the highest degree. We can explore the impact of removing Vertex 10 from the network by deleting the 10th row and column from $\mathbf{c}$ and recomputing. The resulting graphs are shown in Figure \ref{fig:SecondGraph}.
\begin{figure}[htpb]
\centering
\subfigure[Worst Graph]{\includegraphics[scale=0.3]{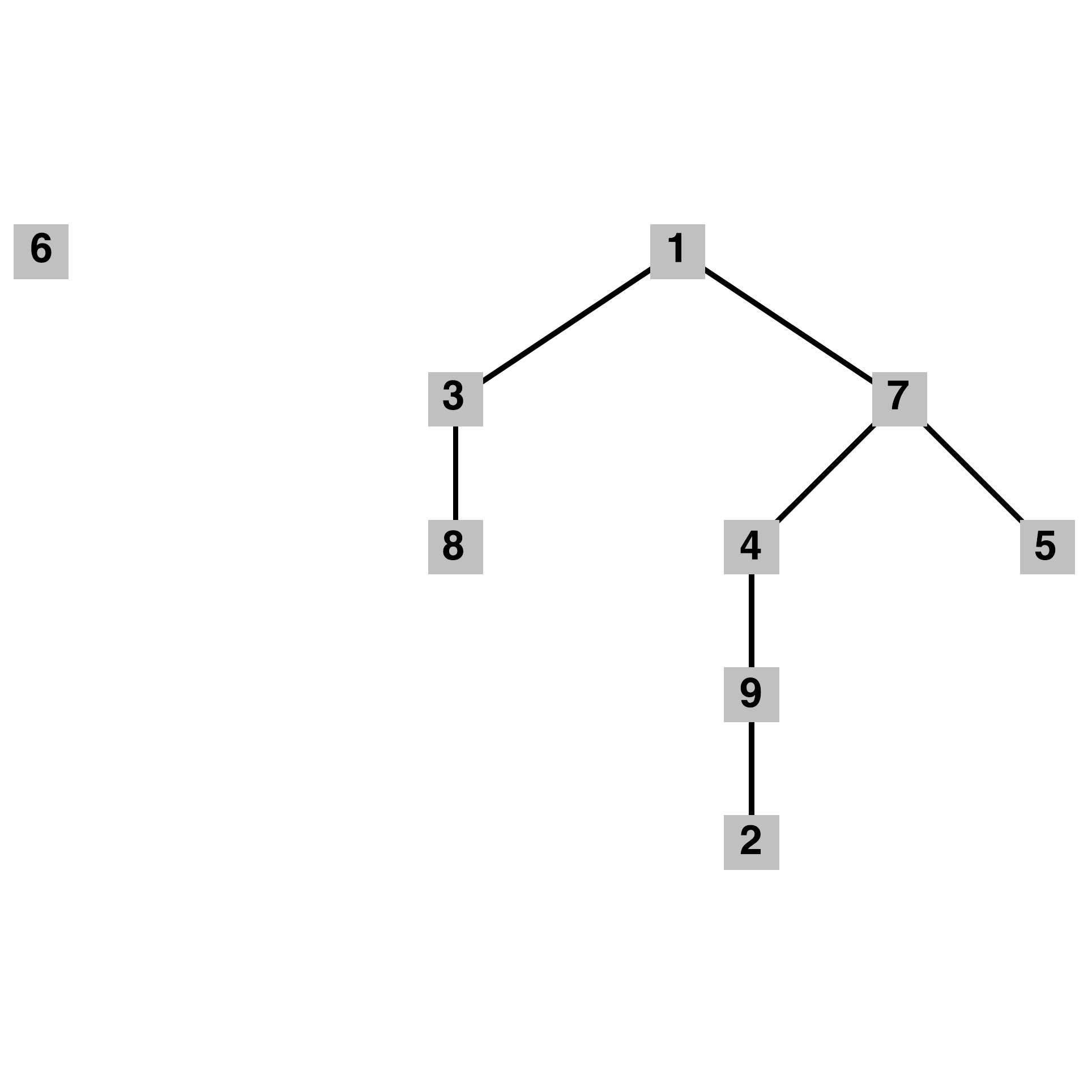}}
\subfigure[Best Graph]{\includegraphics[scale=0.3]{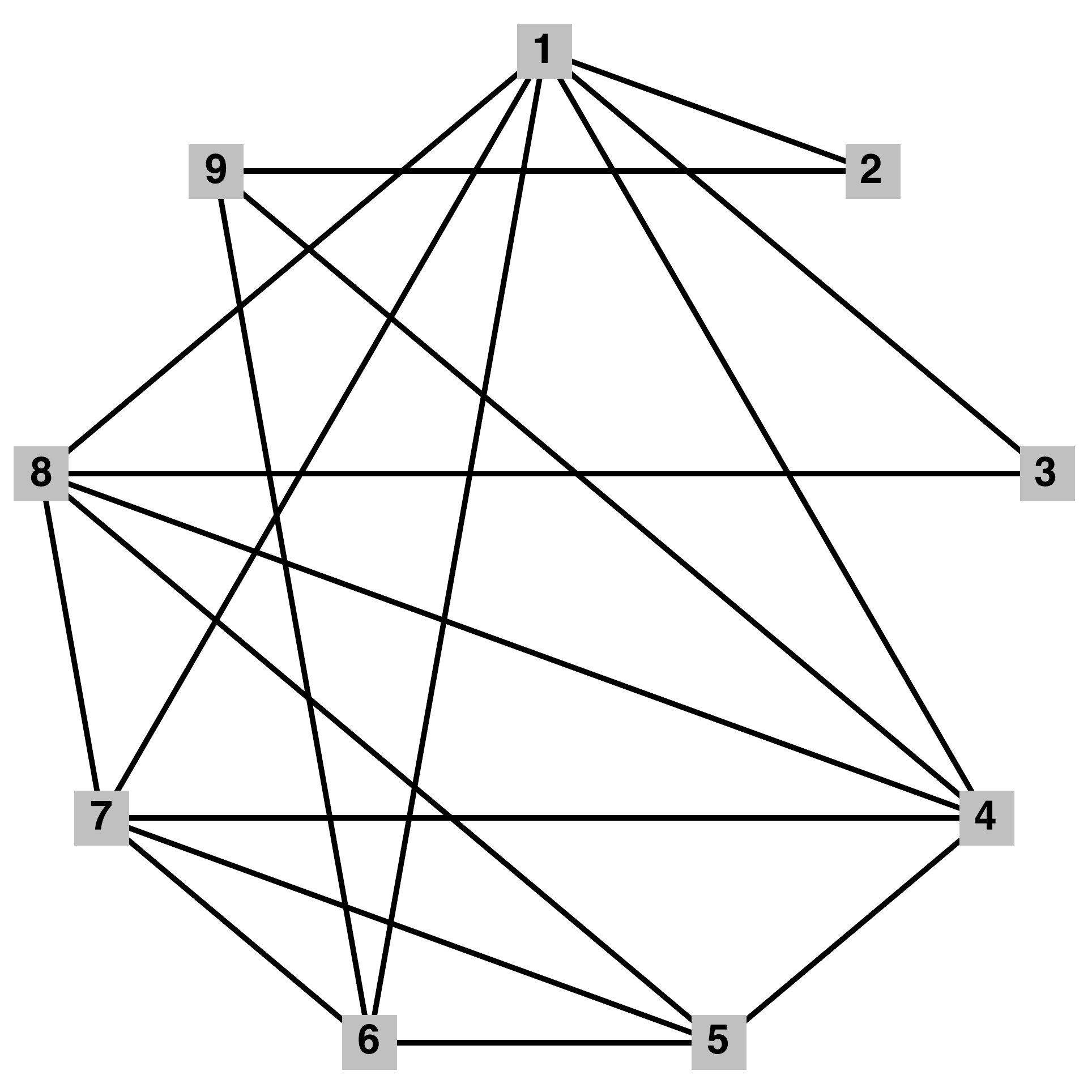}}
\caption{The worst stable graph and best centrally coordinated graph after removing Player 10.}
\label{fig:SecondGraph}
\end{figure}
In this case, the uncoordinated network fragments into a tree and a single, isolated vertex. Furthermore, the new payoffs are $501$ reward units for the stable network and $789$ reward units for the new centrally coordinated network. The new price of anarchy (as a ratio) is $\sim 63\%$, suggesting we have seriously impacted the ability of the network to achieve results as good as a centrally coordinated network.

However, suppose it was easier to remove Vertex 1. If we execute this mission, the resulting graphs are shown in Figure \ref{fig:ThirdGraph}.
\begin{figure}[htpb]
\centering
\subfigure[Worst Graph]{\includegraphics[scale=0.3]{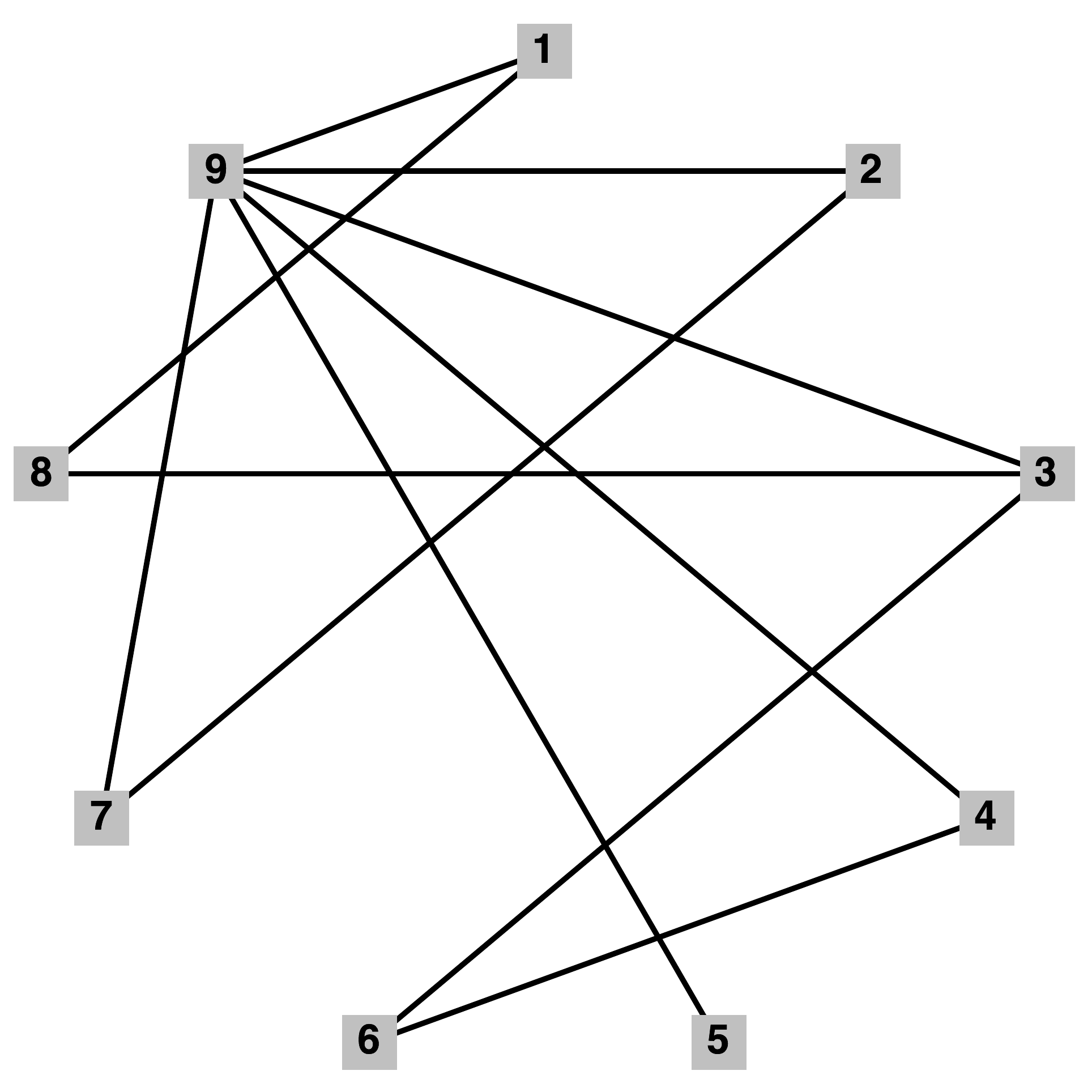}}
\subfigure[Best Graph]{\includegraphics[scale=0.3]{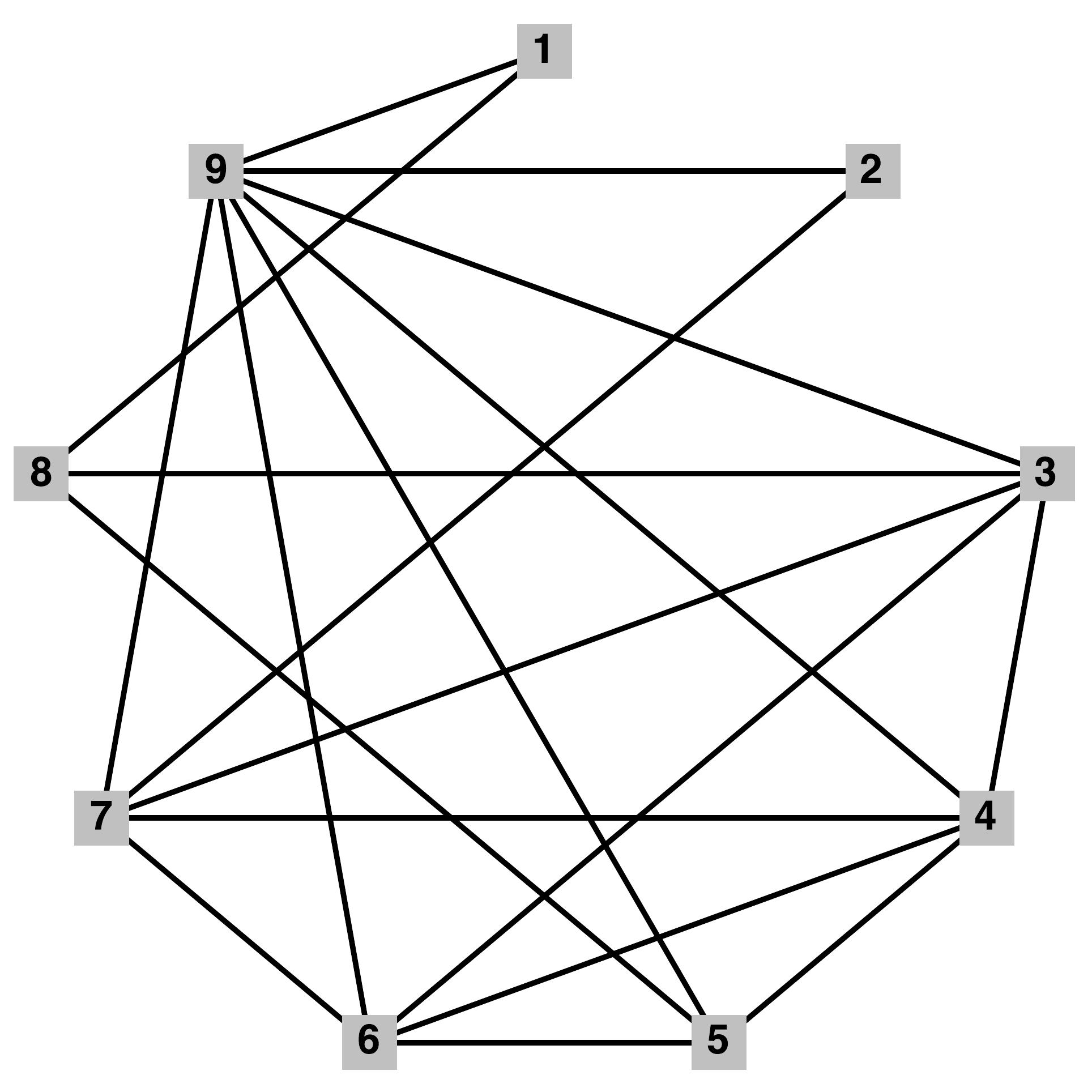}}
\caption{The worst stable graph and best centrally coordinated graph after removing Player 1.}
\label{fig:ThirdGraph}
\end{figure}
In this case, the uncoordinated graph does not fragment at all and more importantly the new payoffs are 899 reward units in the uncoordinated stable graph and 1220 reward units in the centrally coordinated graph. The new price of anarchy (as a ratio) is $\sim74\%$. Removing Vertex 1 actually improves the relative performance of the network with respect to a centrally coordinated network. Thus we might conclude that given the opportunity to remove a vertex, we should choose to remove Vertex 10 rather than Vertex 1, even though these vertices have similar network characteristics. It is interesting to note in this example there is little correlation between the traditional measures of vertex importance and their impact on the resulting network's price of anarchy. Eigenvector centrality and vertex degree for the original worst stable graph (the assumed initial condition) and the resulting change in price of anarchy are summarized in the table below. Price of anarchy difference is computed as the original price of anarchy ($\sim 72\%$) minus new price of anarchy once a vertex is removed. We also include the Communal Change in Utility. This is the difference in the communal objective function in the original graph (10 players) and the communal objective function in the graph that results from removing a vertex (9 players). The computation is done for each player who could be removed. 
\begin{table}[htbp]
\centering
\begin{tabular}{|c|c|c|c|c|}
\hline
Removed Vertex	& Degree	& Eig. Centrality	& POA Diff & Communal Utility Change\\
\hline
1	&2	&0.070066565	&-0.012608178 & 178\\
\hline
2	&2	&0.085041762	&0.026391872 & 153\\
\hline
3	&3	&0.120257982	&0.065375238 & 285\\
\hline
4	&3	&0.115436794	&0.009530895 & 190\\
\hline
5	&2	&0.093603947	&0.011948301 & 193\\
\hline
6	&1	&0.063208915	&-0.03956057 & 42\\
\hline
7	&3	&0.092026999	&-0.072036296 & 213\\
\hline
8	&2	&0.102880448	&-0.05390475 & 221\\
\hline
9	&2	&0.066087279	&-0.003131446 & 103\\
\hline
10	&6	&0.191389311	&0.089296079 & 576\\
\hline
\end{tabular}
\caption{Summary Table for traditional network importance measures and the corresponding impact on price of anarchy.}
\label{tab:SummaryStats}
\end{table}
The information in the table is illustrated in Figures \ref{fig:POAvsStuff} and \ref{fig:UtilvsStuff}.
\begin{figure}[htbp]
\centering
\subfigure[Degree]{\includegraphics[scale=0.3125]{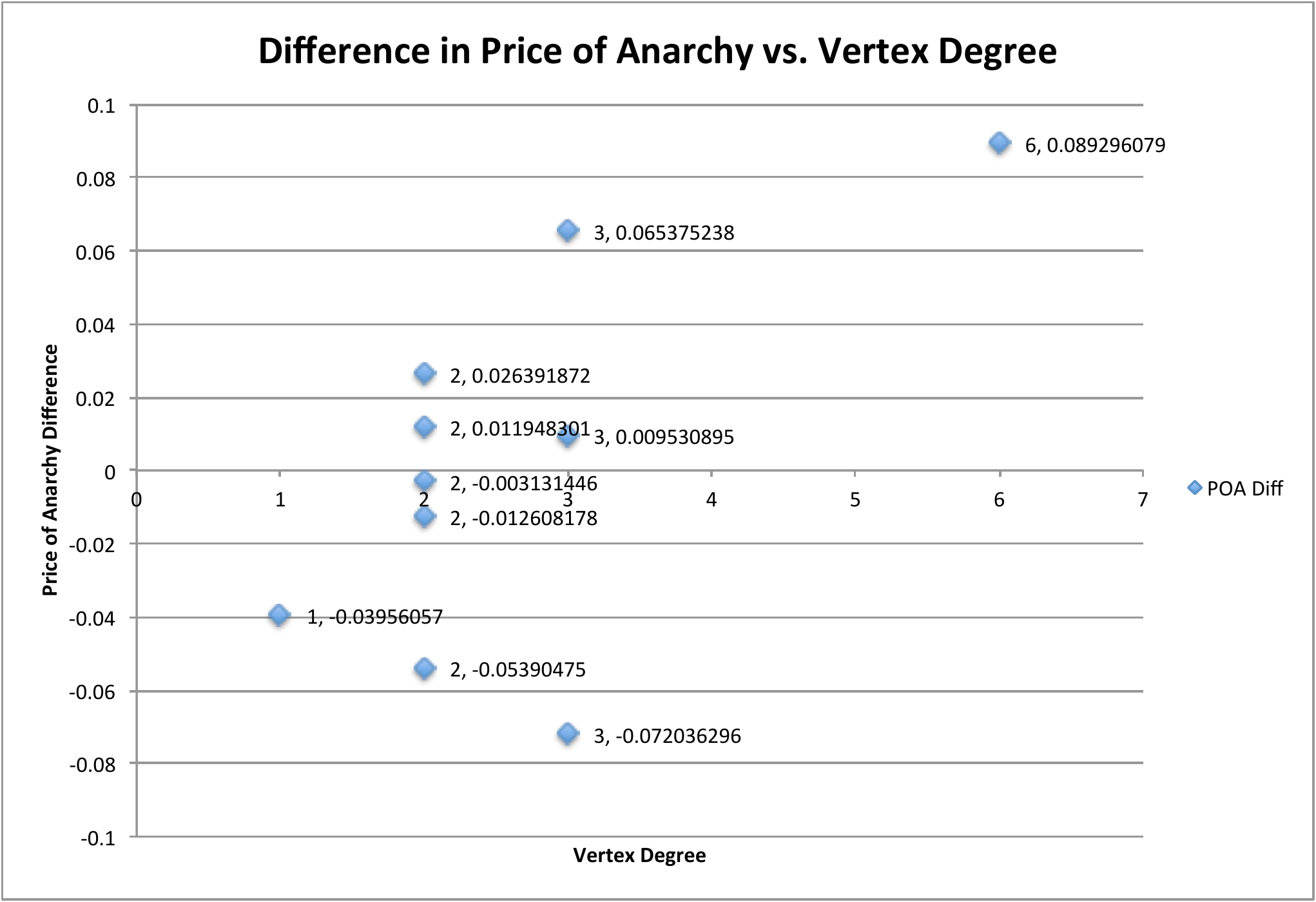}}
\subfigure[Eig. Cent.]{\includegraphics[scale=0.35]{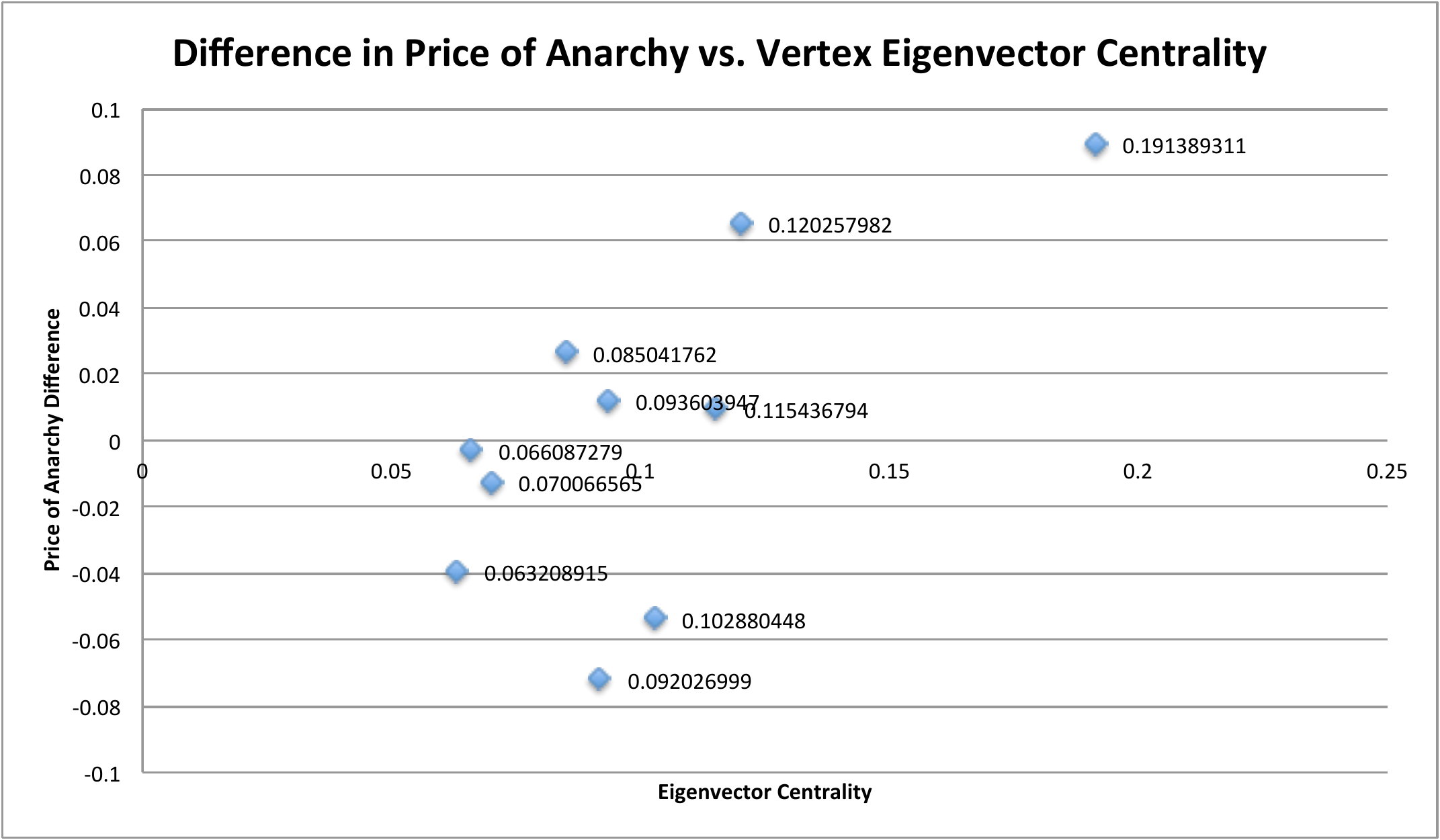}}
\caption{The price of anarchy change as a function of various vertex importance measures is illustrated.}
\label{fig:POAvsStuff}
\end{figure}
It is interesting to note the lack of correlation in these plots. Obviously this is anecdotal evidence, but suggests interesting future work in so far as price of anarchy change may be a new way to measure the importance of a vertex in a social network of interest.
\begin{figure}[htbp]
\centering
\subfigure[Degree]{\includegraphics[scale=0.35]{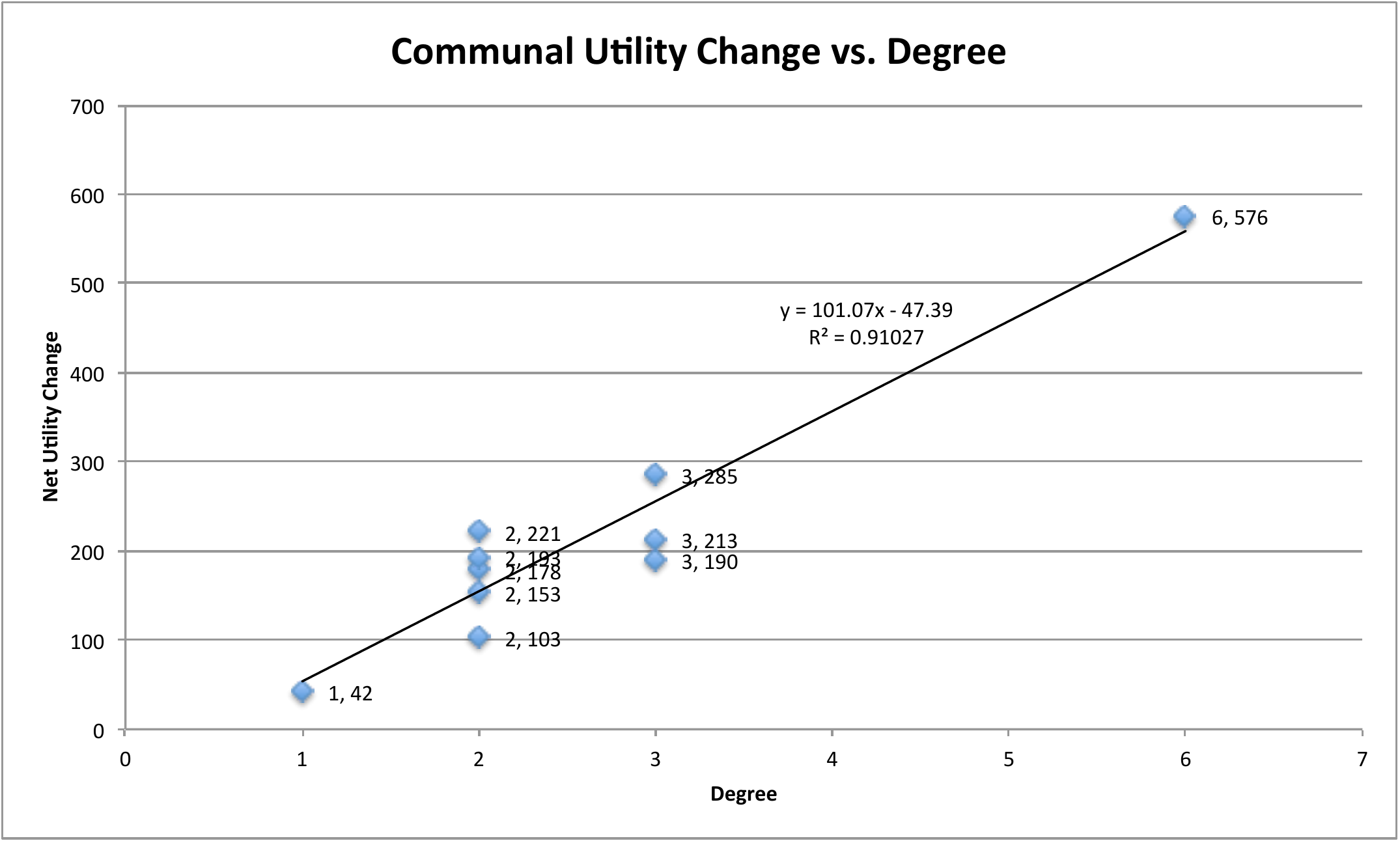}}
\subfigure[Eig. Cent.]{\includegraphics[scale=0.345]{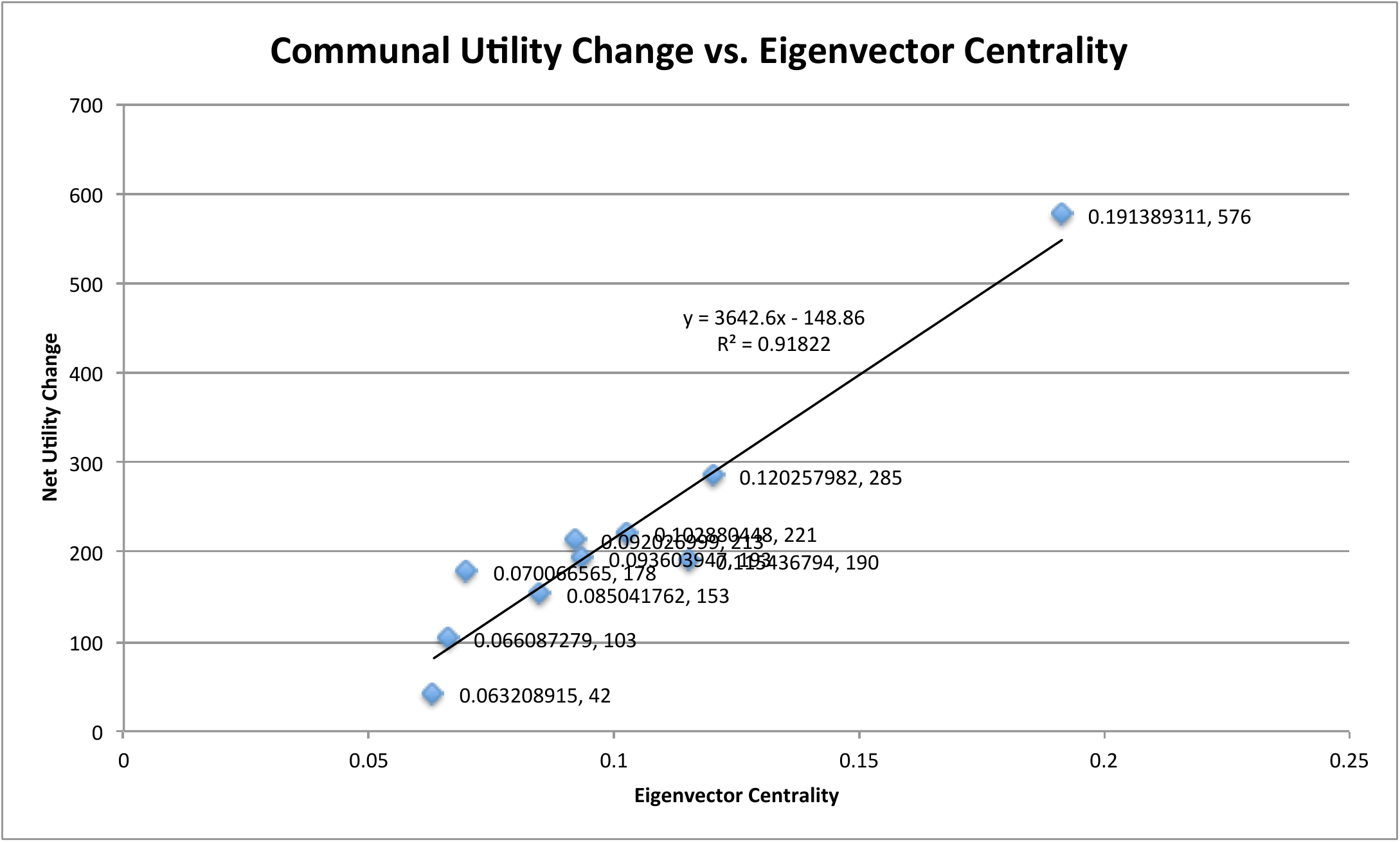}}
\caption{The net utility change as a function of various vertex importance measures is illustrated}.
\label{fig:UtilvsStuff}
\end{figure}
We contrast these plots to the plots in Figure \ref{fig:UtilvsStuff}, where we illustrate the relationship between traditional network metrics and the change in the communal utility. There is clearly a substantial correlation between the degree of a vertex, it's eigenvector centrality and the extent to which the removal of this vertex impacts the communal utility. This is expected since for this game, we can compute the communal utility change when vertex $i$ is removed ($\Delta U_{i}$) as shown in Equation \ref{eqn:CUC}.
\begin{equation}
\label{eqn:CUC}
\Delta U_{i}=-\sum_{j}s_{ij}s_{ji}(c_{ij}+c_{ji})
\end{equation}
For each vertex $i$, when it is the removed vertex, $\Delta U_{i}$ simply sums up the lost utility for each vertex $i$ and $j$ for each link that existed in the network before removal.  When $s_{ij}=0$ or $s_{ji}=0$, this implies the link did not exist and hence, no contribution is made to $\Delta U_{i}$.  Alternatively, when $s_{ij}=s_{ji}=1$, the contribution made to $\Delta U_{i}$ is the sum of the utilities received by vertex $i$ and $j$, which is $-(c_{ij}+c_{ji})$.  Since, $s_{ij}=1$ only if $c_{ij}<0$ and $s_{ji}=1$ only if $c_{ji}<0$, this means that when the link does exist $-(c_{ij}+c_{ji})>0$.  Hence, only positive contributions can be made to $\Delta U_{i}$ for each link it has and hence it must be positively correlated with the degree of vertex $i$.  However, since utility received by a link may be quite asymmetric when $c_{ij}<<c_{ji}$, the communal utility change may differ greatly between two nodes with the same degree.

This example illustrates the importance of understanding what network metrics mean in a given application. Because each player (and thus the community of players) derives benefit from being connected to specific players in this game, it is relatively straightforward to see that network metrics like eigenvector centrality and degree will be highly correlated to value functions like communal utility change; that is, how much removing a single player will decrease the total payoff to the network. On the other hand, it is clear that there is little relationship between the price of anarchy difference and the network centrality measure of a vertex. Thus, if the goal is to degrade the network's ability to accumulate utility, then using eigenvector centrality as a proxy measure \textit{in the link bias game} should be acceptable. However, if the objective is to cause the network to function in the least centralized way possible, then eigenvector centrality may not be as good a proxy measure.

The is most clear in the case when we attempt to optimize both of these objectives in a kill or capture mission at once. Consider Table \ref{tab:SummaryStats}: Suppose a decision maker's objective is to simultaneously maximize the decrease in communal benefits to the network and increase the price of anarchy as much as possible. Clearly, Vertex 10 is the ideal target; in fact it is the Pareto optimal solution for that multi-criteria optimization problem. If Vertex 10 is not accessible, then Vertex 3 is the next obvious target. It too has an undominated payoff pair for that problem (the net benefit decrease is 285 reward units and the price of anarchy decreases from $72\%$ to $66\%$). Interestingly, these two vertices have the highest eigenvector centrality but not necessarily the uniquely highest degree (in the case of Vertex 3). If Vertex 3 also cannot be killed or captured, then the problem becomes more complicated. Eigenvector centrality is no longer a good proxy measure since the change in price of anarchy is not consistent with the change in communal utility.

\section{Simulation}
Networks that are pairwise stable for a game are pairwise stable because no player has an incentive to drop a link and no two players have an incentive to add an additional link.  The price of anarchy measures how much worse the worst possible stable graph is from the best for a particular game.  However, this analysis does not offer insight into the actual formation of games. The simulation of network formation can offer some additional insight. Moreover, for exceptionally large games solving the Integer Programming Problems associated with computing the price of anarchy may be difficult. While it is relatively easy to solve Problem \ref{eqn:LinkBiasBestGraph} or Problem \ref{model: POA_best_1}, it may be difficult to solve Problem \ref{model: WorstGraphIP}.

\subsection{Methodology}
\label{sec: methodology}
In this subsection, we outline the methodology for simulating the formation of a network.  We again denote the current graph as $\mathbf{x}$ where:
\begin{equation}
x_{ij} =x_{ji} =
\begin{cases}
1 & \text{if node $i$ is linked to node $j$}\\
0 & \text{else}
\end{cases}
\label{eqn:xdef2}
\end{equation}
We denote the matrix $\mathbf{P}$ as the matrix representing potential links:
\begin{equation}
P_{ij} =P_{ji} =
\begin{cases}
1 & \text{if node $i$ \emph{could} link to node $j$}\\
0 & \text{else}
\end{cases}
\label{eqn:pdef1}
\end{equation}
A link $ij$ is a potential link if:
\begin{enumerate}
  \item Link $ij$ currently does not exist $(x_{ij}=x_{ji}=0)$
  \item Node $i$ could benefit from linking to node $j$ $(i.e., \sum_{j}x_{ij}<d_{i})$
  \item Node $j$ could benefit from linking to node $i$ $(i.e., \sum_{i}x_{ji}<d_{j})$
\end{enumerate}
Now, we introduce our simulation algorithm:
\begin{enumerate}
  \item Initialize $x_{ij}=0$, $P_{ij}=1$ for all links $ij$
  \item While there exists a potential link (i.e., $\sum_{ij}P_{ij}>0$):
  \begin{enumerate}
    \item Randomly select a potential link $ij$ (i.e. randomly choose a pair $(i,j)$ such that $P_{ij}=1$)
    \item Add link $ij$ to the graph, set $x_{ij}=1$,$x_{ji}=1$
    \item Delete link $ij$ from the potential links, set $P_{ij}=0$,$P_{ji}=0$
    \item If node $i$ cannot benefit from linking to any more nodes $(\sum_{k}x_{ik} \geq d_{i})$, then remove all of their potential links.  That is, set $P_{ik}=0$ and $P_{ki}=0$ for all $k$
    \item If node $j$ cannot benefit from linking to any more nodes $(\sum_{k}x_{jk} \geq d_{j})$, then remove all of their potential links.  That is, set $P_{jk}=0$ and $P_{kj}=0$ for all $k$
  \end{enumerate}
\end{enumerate}

\subsection{Numerical Example}
Suppose that we want the degree sequence of a stable graph that results from playing the game described in Theorem \ref{thm:GeneralArbitraryDist} to have a power law degree distribution.  We embed this into the objectives of the players, so the resulting graph has the proper distribution. Let $n=100$ players attempt to minimize their cost function
\begin{displaymath}
f_{i}(\eta_{i}(g))=f(\eta_{i}(g) - k_{i})=|\eta_{i}(g) - k_{i}|
\end{displaymath}
where the parameter $k_i$ for each player and the degree distribution of $k_{i}$ values may be found in Table \ref{fig:ktable}.  The distribution of the $k_{i}$ values form an approximate (with rounding to integers) power law distribution as illustrated in the histogram in Figure \ref{fig:HistogramPowerLaw}. We note that this degree sequence is graphical. Thus the solution to Problem \ref{model: optimization 1B2} yields an objective function that is 0.


\begin{figure}[htbp]
\centering
\subfigure{\begin{tabular}[b]{|c|c|c|}
\hline
Node(s) & Number of Nodes   & $k_{i}$  \\
\hline
1-75    &   75  &   1\\
\hline
76-89   &   14  &   2\\
\hline
90-94   &   5   &   3\\
\hline
95-96   &   2   &   4\\
\hline
97      &   1   &   5\\
\hline
98      &   1   &   6\\
\hline
99      &   1   &   7\\
\hline
100     &   1   &   8\\
\hline
\end{tabular}}
\caption{$k_{i}$ values and degree distribution}
\label{fig:ktable}
\end{figure}
\begin{figure}[htbp]
\centering
\includegraphics[scale=0.4]{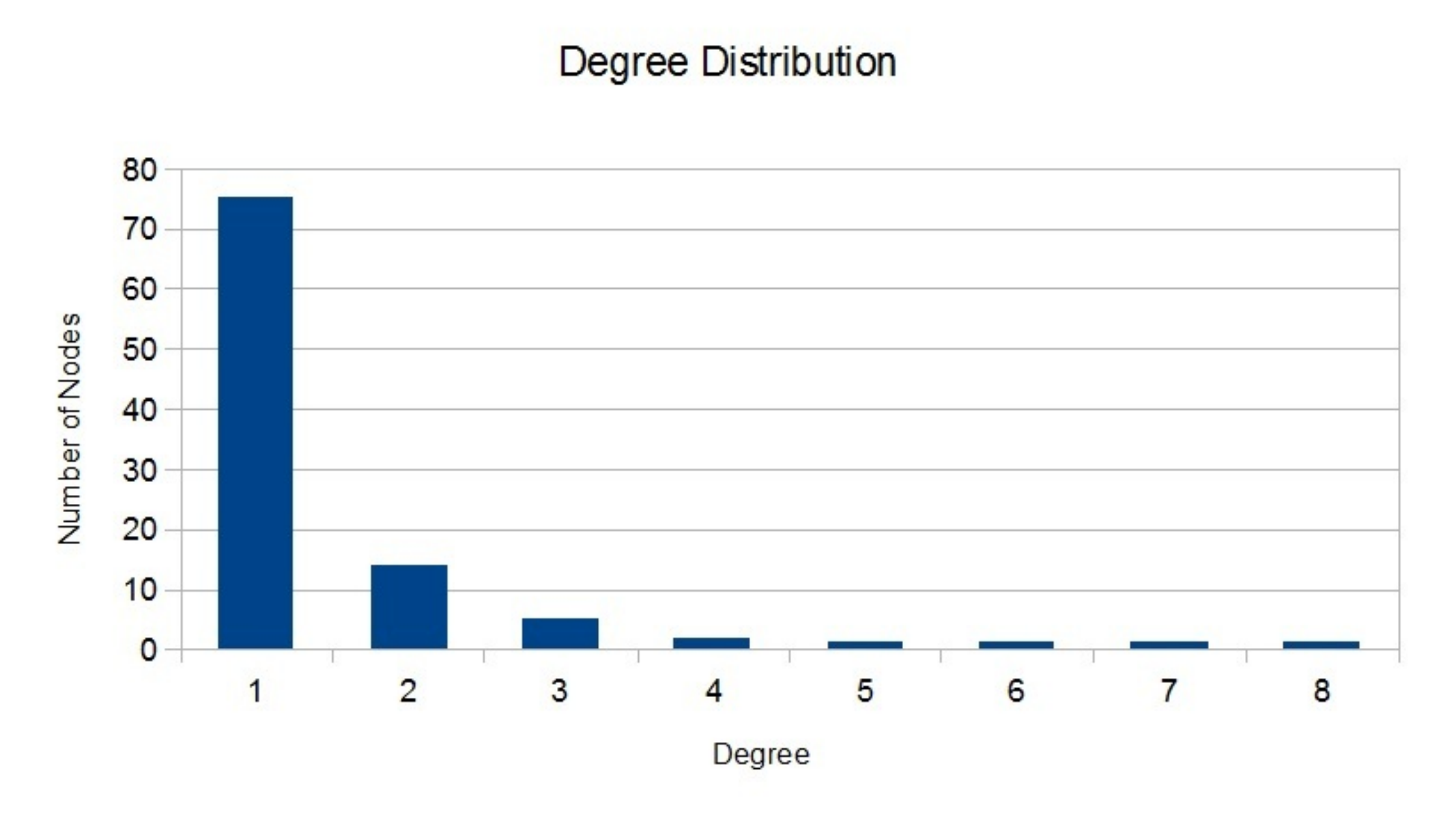}
\caption{The $k_{i}$ values follow an approximate power law distribution}
\label{fig:HistogramPowerLaw}
\end{figure}
%
The price of anarchy is the difference of the objective function value from \emph{the worst graph} (Problem (\ref{model: WorstGraphIP})) to \emph{the best graph} (Problem (\ref{model: optimization 1B2})).


We simulated this game $100$ times using the methodology from Section \ref{sec: methodology} to find further insight.  In Figure \ref{fig:distnStats} we show the simulation statistics for the degree distribution.  For example, we see in Figure \ref{fig:ktable} that five nodes would optimize their objectives if they had a degree of three.  In Figure \ref{fig:distnStats}, we see that in all of the simulations, the minimum number of nodes with a degree of three was four and the maximum number of nodes with a degree of three was six.  Since, the twenty-fifth and seventy-fifth percentile was five, we know that in more than half of the simulation runs, five nodes had a degree of three and optimized their objective.  In the other half of the runs, there was only one node too many or one node too few with a degree of three, this small variation indicates that most nodes get close to their optimum degree resulting in a rather small price of anarchy.  The simulation statistics are also visually presented as a box plot in Figure \ref{fig:boxplot}
\begin{figure}[htbp]
\centering
\subfigure{\begin{tabular}[b]{|c|c|c|c|c|c|}
\hline
Degree & min & $25^{th}$ & median & $75^{th}$ & max\\
\hline
1&  75 &   75  &    75   &  75  &   76 \\
\hline
2&  13 &   14  &    14   &  14  &   15 \\
\hline
3&  4  &   5   &     5   &  5   &   6  \\
\hline
4&  1  &   2   &     2   &  3   &   4 \\
\hline
5&  0  &   1   &     1   &  2   &   3 \\
\hline
6&  0  &   1   &     1   &  2   &   2 \\
\hline
7&  0  &   0   &     1   &  1   &   2 \\
\hline
8&  0  &   0   &     0   &  1   &   1 \\
\hline
\end{tabular}}
\caption{Degree Distribution Simulation Results}
\label{fig:distnStats}
\end{figure}
\begin{figure}[htbp]
\centering
\includegraphics[scale=0.4]{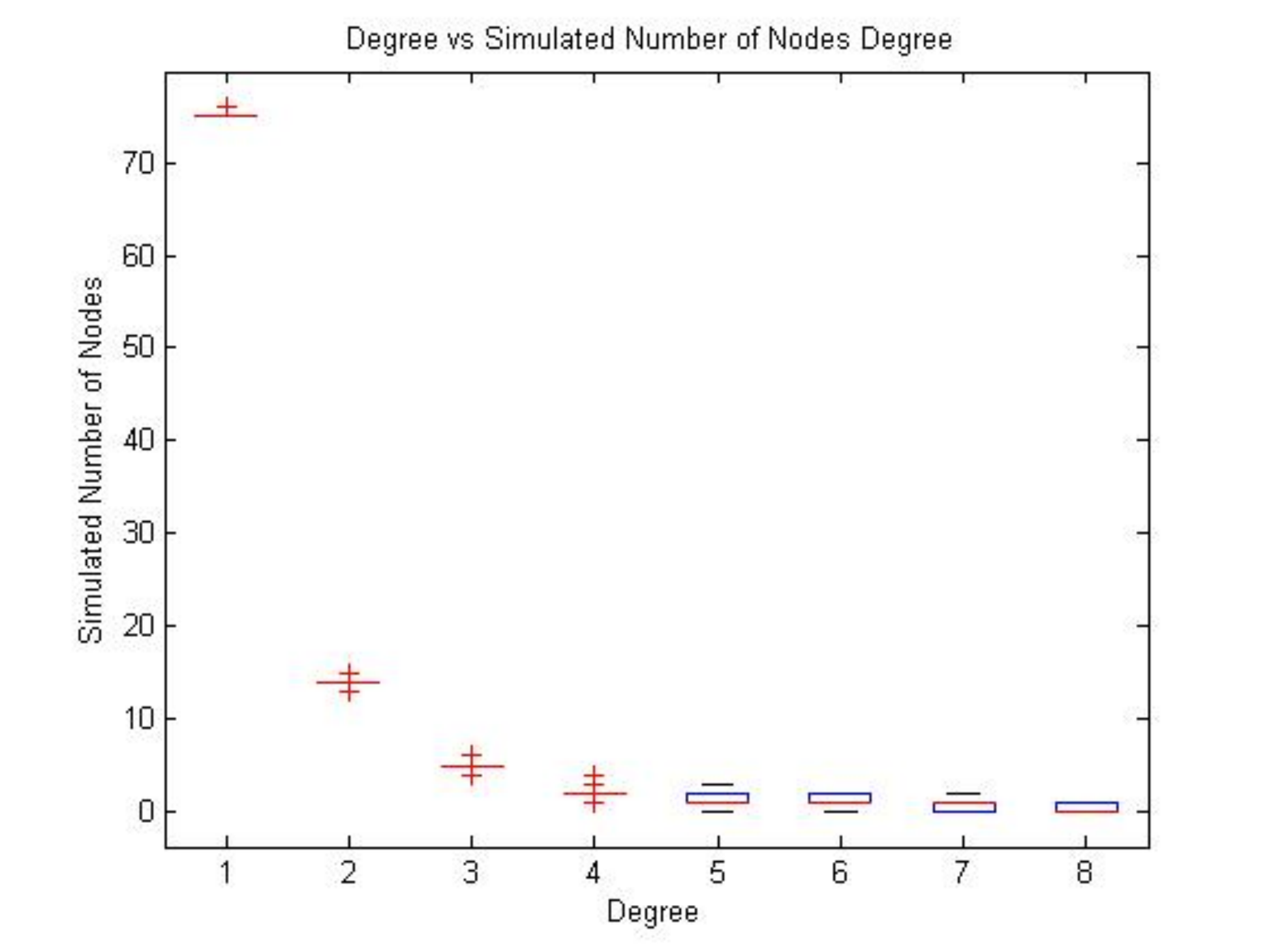}
\label{fig:boxplot}
\end{figure}
In each simulation run, we calculated the price of anarchy and plotted the results as a histogram in Figure \ref{fig:HistogramPOA}. Note the largest price of anarchy is 10, suggesting that this is the \textit{true} price of anarchy for the system. That is, the worst possible outcome of communal utility in competitive play minus the best outcome in centralized decision making (which is zero, since the degree sequence given is graphical).
\begin{figure}[htbp]
\centering
\includegraphics[scale=0.6]{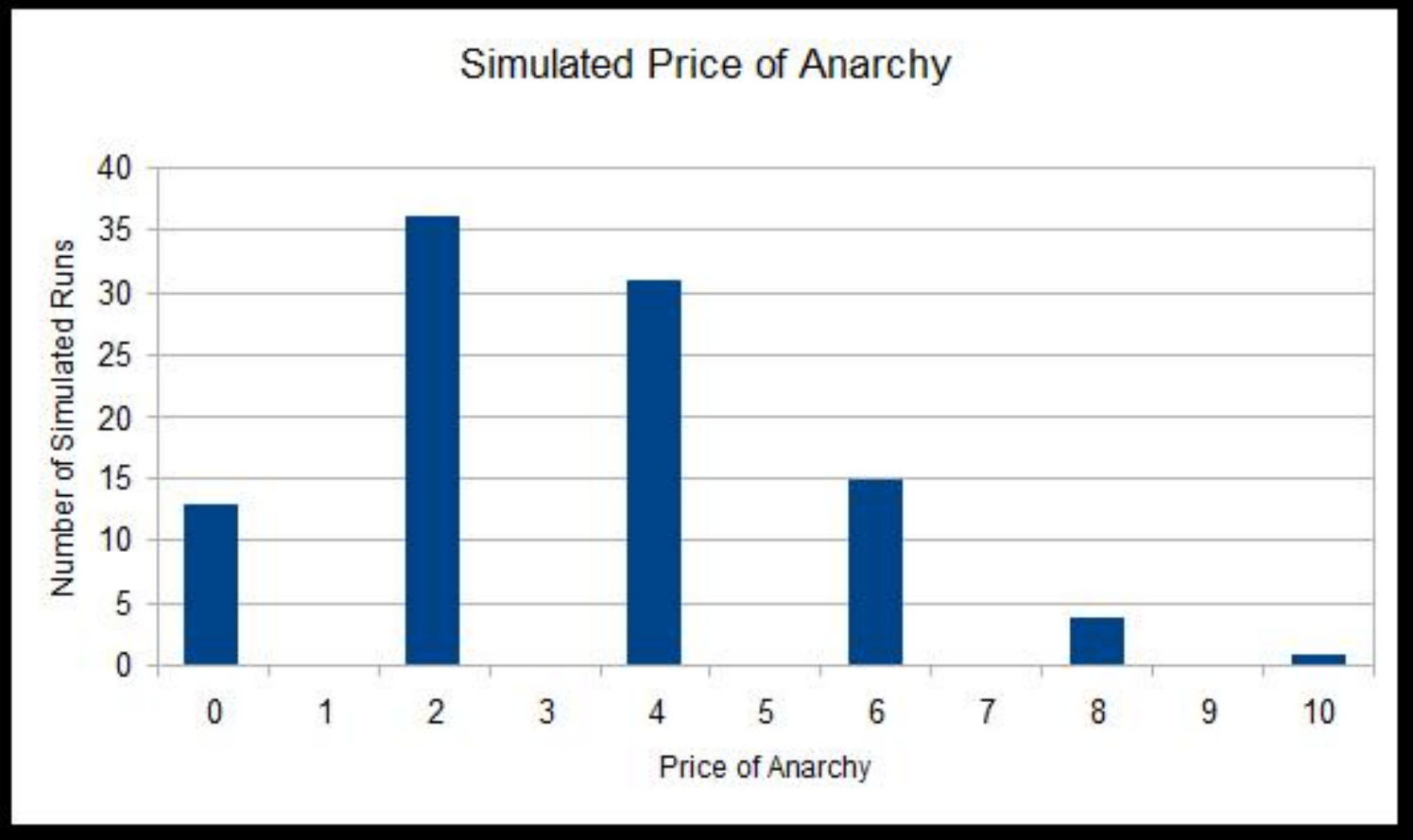}
\caption{The empirical distribution of the price of anarchy. }
\label{fig:HistogramPOA}
\end{figure}
The price of anarchy is rather low in most simulation runs.  We investigated the distribution of the contributions to the price of anarchy.  As shown in Figure \ref{fig:POAbyDegree}, the contributions to the price of anarchy were of higher magnitude and made more often by nodes with a greater $k_{i}$ value.  This means that players that desired more links were more often more unhappy than other nodes who desired a lower degree.  That being said, Figure \ref{fig:POAbyDegree} and Figure \ref{fig:HistogramPOA} show that most player's attain their desired degree in most of the simulations.
\begin{figure}[htbp]
\centering
\subfigure{\begin{tabular}[b]{|c|c|c|c|c|c|c|c|c|}
\hline
Degree & min & $10^{th}$ & $25^{th}$ & median & $75^{th}$ & $90^{th}$& $95^{th}$ & max\\
\hline
1&  0  &   0  &   0   &    0    &  0   &   0   &   0   &   0\\
\hline
2&  0  &   0  &   0   &    0    &  0   &   0   &   0   &   1\\
\hline
3&  0  &   0  &   0   &    0    &  0   &   0   &   0   &   1\\
\hline
4&  0  &   0  &   0   &    0    &  0   &   0   &   0   &   2\\
\hline
5&  0  &   0  &   0   &    0    &  0   &   0   &   1   &   3\\
\hline
6&  0  &   0  &   0   &    0    &  1   &   1   &   2   &   2\\
\hline
7&  0  &   0  &   0   &    0    &  2   &   3   &   4   &   6\\
\hline
8&  0  &   0  &   0   &    2    &  3   &   4   &   4   &   6\\
\hline
\end{tabular}}
\caption{Price of Anarchy Contributions by Degree}
\label{fig:POAbyDegree}
\end{figure}

\section{Conclusion and Future Directions}
In this paper we have studied game theoretic explanations for the formation of networks. We have shown that networks with interesting structures may emerge as the result of interactive play between competitive individuals. We have also illustrated a method for computing the price of anarchy for two types of network games and illustrated the application of this calculation to kill or capture operations on violent extremist and criminal groups. We concluded by showing how simulation could be used for evaluating larger networks of interest when optimization problems become infeasible. There are several future directions of work.

We would like to investigate more complex and more realistic games that will, ideally produce more lifelike behavior. To do this, we must be able to infer the objective functions of the various players. In practice, these are never known and must be estimated, and it may be quite difficult to even estimate them.  It is simply not possible to perfectly observe an entire graph at a single instance in time.  Even if it were possible, relationships often cannot be characterized by a single binary variable.  Nonetheless, each time a graph is observed, information about the game can be extracted. Investigating the process of identifying these objective functions is fundamental to extending these techniques to real world situations.

In addition to this, more theoretical work can be done. Obtaining theoretical bounds on the price of anarchy and on the change in the price of anarchy as a result of vertex removal is in keeping with the literature on network formation games. In addition to this, investigation of more complex games in which constraints on players influence strategies is of interest. The generalized Nash equilibrium games \cite{FFP07} may be of interest in this case and will provide a richer context in which to model player behavior and therefore more requirements on our ability to infer player objectives and constraints from observed graphs.

\bibliographystyle{IEEEtran}
\bibliography{Biblio-Database2}

\begin{thebibliography}{10}
\providecommand{\url}[1]{#1}
\csname url@samestyle\endcsname
\providecommand{\newblock}{\relax}
\providecommand{\bibinfo}[2]{#2}
\providecommand{\BIBentrySTDinterwordspacing}{\spaceskip=0pt\relax}
\providecommand{\BIBentryALTinterwordstretchfactor}{4}
\providecommand{\BIBentryALTinterwordspacing}{\spaceskip=\fontdimen2\font plus
\BIBentryALTinterwordstretchfactor\fontdimen3\font minus
  \fontdimen4\font\relax}
\providecommand{\BIBforeignlanguage}[2]{{%
\expandafter\ifx\csname l@#1\endcsname\relax
\typeout{** WARNING: IEEEtran.bst: No hyphenation pattern has been}%
\typeout{** loaded for the language `#1'. Using the pattern for}%
\typeout{** the default language instead.}%
\else
\language=\csname l@#1\endcsname
\fi
#2}}
\providecommand{\BIBdecl}{\relax}
\BIBdecl

\bibitem{Doz11}
K.~Dozier, ``Cia following twitter, facebook,'' Associated Press:
  http://news.yahoo.com/ap-exclusive-cia-following-twitter-facebook-081055316.html,
  November 2011.

\bibitem{CSW05}
P.~J. Carrington, J.~Scott, and S.~Wasserman, \emph{{Models and Methods in
  Social Network Analysis}}.\hskip 1em plus 0.5em minus 0.4em\relax {Cambridge
  University Press}, 2005.

\bibitem{KY08}
D.~Knoke and S.~Yang, \emph{{Social Network Analysis}}, ser. Quantitative
  Applications in the Social Sciences.\hskip 1em plus 0.5em minus 0.4em\relax
  SAGE Publications, 2008, no. 154.

\bibitem{AH10}
S.~Asur and B.~A. Huberman, ``Predicting the future with social media,'' AxXiV:
  1003.5699v1, Tech. Rep., March 2010.

\bibitem{BDK07}
L.~Backstrom, C.~Dwork, and J.~Kleinberg, ``Wherefore art thou r3579x?
  anonymized social networks, hidden patterns and structural steganography,''
  in \emph{Proc. 16th World Wide Web Conference}, 2007.

\bibitem{FCF11}
A.~Friggeri, G.~Chelius, and E.~Fleury, ``Fellows: Crowd-sourcing the
  evaluation of an overlapping community model based on the cohesion measure,''
  in \emph{Proc. Interdisciplinary Workshop on Information and Decision in
  Social Networks}.\hskip 1em plus 0.5em minus 0.4em\relax Massachusetts
  Institute of Technology: Laboratory for Ilnformation and Decision Systems,
  May 31 - Jun 1 2011.

\bibitem{OT11}
A.~Olshevsky and J.~N. Tsitsiklis, ``Convergence speed in distributed consensus
  and averaging,'' \emph{SIAM Review}, vol.~53, no.~4, pp. 747--772, 2011.

\bibitem{JGN01}
E.~M. Jin, M.~Girvan, and M.~E.~J. Newman, ``Structure of growing social
  networks,'' \emph{Physical Review}, vol.~64, no. 046132, 2001.

\bibitem{SSG11}
A.~C. Squicciarini, S.~Sundareswaran, and C.~Griffin, ``A game theoretical
  perspective of users' registration in online social platforms,'' in
  \emph{Accepted to Third IEEE International Conference on Privacy, Security,
  Risk and Trust}, October 9 - 11 2011.

\bibitem{YD08}
W.~S. Yang and J.~B. Dia, ``Discovering cohesive subgroups from social networks
  for targeted advertising,'' \emph{Expert Syst. Appl.}, vol.~34, no.~3, pp.
  2029--2038, 2008.

\bibitem{barabasi1999a}
A.~Barabasi and R.~Albert, ``{Emergence of scaling in random networks},''
  \emph{Science}, vol. 286, no. 5439, p. 509, 1999.

\bibitem{newman2003}
M.~Newman, ``{The structure and function of complex networks},'' \emph{SIAM
  Review}, vol.~45, pp. 167--256, 2003.

\bibitem{dorogovtsev2002}
S.~Dorogovtsev and J.~Mendes, ``{Evolution of networks},'' \emph{Advances in
  Physics}, vol.~51, no.~4, pp. 1079--1187, 2002.

\bibitem{albert2002}
R.~Albert and A.~Barab{\'a}si, ``{Statistical mechanics of complex networks},''
  \emph{Reviews of modern physics}, vol.~74, no.~1, pp. 47--97, 2002.

\bibitem{doyle2005}
J.~Doyle, D.~Alderson, L.~Li, S.~Low, M.~Roughan, S.~Shalunov, R.~Tanaka, and
  W.~Willinger, ``{The "robust yet fragile" nature of the Internet},''
  \emph{Proceedings of the National Academy of Sciences of the United States of
  America}, vol. 102, no.~41, p. 14497, 2005.

\bibitem{myerson1977}
R.~Myerson, ``{Graphs and cooperation in games},'' \emph{Mathematics of
  Operations Research}, vol.~2, no.~3, pp. 225--229, 1977.

\bibitem{jackson1996}
M.~Jackson and A.~Wolinsky, ``{A strategic model of social and economic
  networks},'' \emph{Journal of economic theory}, vol.~71, no.~1, pp. 44--74,
  1996.

\bibitem{dutta1997}
B.~Dutta and S.~Mutuswami, ``{Stable networks},'' \emph{Journal of Economic
  Theory}, vol.~76, no.~2, pp. 322--344, 1997.

\bibitem{goyal2003}
S.~Goyal and S.~Joshi, ``{Networks of collaboration in oligopoly},''
  \emph{Games and Economic behavior}, vol.~43, no.~1, pp. 57--85, 2003.

\bibitem{LichterGriffinFriesz2011}
S.~Lichter, C.~Griffin, and T.~Friesz, ``A game theoretic perspective on
  network topologies,'' Submitted to Computational and Mathematical
  Organization Theory (under revision) [http://arxiv.org/abs/1106.2440],
  February 2011.

\bibitem{LGF11}
------, ``Link biased strategies in network formation games,'' in \emph{In
  Proc. 3rd IEEE Conference on Social Computing}, MIT, Boston, MA, October 9-11
  2011.

\bibitem{bala2000}
V.~Bala and S.~Goyal, ``{A noncooperative model of network formation},''
  \emph{Econometrica}, vol.~68, no.~5, pp. 1181--1229, 2000.

\bibitem{goyal2006}
S.~Goyal and S.~Joshi, ``{Unequal connections},'' \emph{International Journal
  of Game Theory}, vol.~34, no.~3, pp. 319--349, 2006.

\bibitem{jackson2005}
M.~Jackson and A.~Van Den~Nouweland, ``{Strongly stable networks},''
  \emph{Games and Economic Behavior}, vol.~51, no.~2, pp. 420--444, 2005.

\bibitem{Dies10}
R.~Diestel, \emph{Graph Theory}, 4th~ed., ser. Graduate Texts in
  Mathematics.\hskip 1em plus 0.5em minus 0.4em\relax Springer, 2010.

\bibitem{GR01}
C.~Godsil and G.~Royle, \emph{Algebraic Graph Theory}.\hskip 1em plus 0.5em
  minus 0.4em\relax Springer, 2001.

\bibitem{GY05}
J.~Gross and J.~Yellen, \emph{Graph Theory and its Applications}, 2nd~ed.\hskip
  1em plus 0.5em minus 0.4em\relax Boca Raton, FL, USA: CRC Press, 2005.

\bibitem{jackson2003}
M.~Jackson, ``{A survey of models of network formation: Stability and
  efficiency},'' \emph{Game Theory and Information}, 2003.

\bibitem{belleflamme2004}
P.~Belleflamme and F.~Bloch, ``{Market sharing agreements and collusive
  networks},'' \emph{International Economic Review}, vol.~45, no.~2, pp.
  387--411, 2004.

\bibitem{jackson2008}
M.~Jackson, \emph{{Social and economic networks}}.\hskip 1em plus 0.5em minus
  0.4em\relax Princeton University Press, 2008.

\bibitem{KP99}
E.~Koutsoupias and C.~H. Papadimitriou, ``Worst-case equilibria,'' in \emph{In
  Proc. STACS 99}, 1999.

\bibitem{dubey1986}
P.~Dubey, ``Inefficiency of nash equilibria,'' \emph{Mathematics of Operations
  Research}, pp. 1--8, 1986.

\bibitem{Ros73}
\BIBentryALTinterwordspacing
R.~W. Rosenthal, ``A class of games possessing pure-strategy nash equilibria,''
  \emph{International Journal of Game Theory}, vol.~2, pp. 65--67, 1973,
  10.1007/BF01737559. [Online]. Available:
  \url{http://dx.doi.org/10.1007/BF01737559}
\BIBentrySTDinterwordspacing

\bibitem{GK05}
\BIBentryALTinterwordspacing
G.~Christodoulou and E.~Koutsoupias, ``The price of anarchy of finite
  congestion games,'' in \emph{Proceedings of the thirty-seventh annual ACM
  symposium on Theory of computing}, ser. STOC '05.\hskip 1em plus 0.5em minus
  0.4em\relax New York, NY, USA: ACM, 2005, pp. 67--73. [Online]. Available:
  \url{http://doi.acm.org/10.1145/1060590.1060600}
\BIBentrySTDinterwordspacing

\bibitem{GTR11}
C.~Griffin, K.~Testa, and S.~Racunas, ``An algorithm for searching an
  alternative hypothesis space,'' \emph{IEEE Trans. Sys. Man and Cyber. B},
  vol.~41, no.~3, pp. 772--782, 2011.

\bibitem{britton2006}
T.~Britton, M.~Deijfen, and A.~Martin-L{\\"o}f, ``{Generating simple random
  graphs with prescribed degree distribution},'' \emph{Journal of Statistical
  Physics}, vol. 124, no.~6, pp. 1377--1397, 2006.

\bibitem{milo2003}
R.~Milo, N.~Kashtan, S.~Itzkovitz, M.~Newman, and U.~Alon, ``{On the uniform
  generation of random graphs with prescribed degree sequences},'' \emph{Arxiv
  preprint cond-mat/0312028}, 2003.

\bibitem{del2010}
C.~Del~Genio, H.~Kim, Z.~Toroczkai, and K.~Bassler, ``{Efficient and exact
  sampling of simple graphs with given arbitrary degree sequence},'' \emph{PloS
  one}, vol.~5, no.~4, p. e10012, 2010.

\bibitem{mihail2002}
M.~Mihail and N.~Vishnoi, ``{On generating graphs with prescribed vertex
  degrees for complex network modeling},'' \emph{ARACNE 2002}, pp. 1--11, 2002.

\bibitem{FFP07}
F.~Facchinei, A.~Fischer, and V.~Piccialli, ``On generalized {N}ash games and
  variational inequalities,'' \emph{Operations Research Letters}, vol.~35, pp.
  159--164, 2007.

\end{thebibliography}

\end{document}